\documentclass[11pt]{article}
\usepackage[margin=1in]{geometry}
\usepackage{times}
\usepackage[round,authoryear]{natbib}

\usepackage{amsmath,amssymb,amsthm}
\usepackage{booktabs}
\usepackage{array}
\usepackage{graphicx}
\usepackage{mathtools}
\usepackage{microtype}
\usepackage{placeins}
\usepackage{float}
\usepackage{hyperref}
\usepackage{url}
\hypersetup{colorlinks=true,linkcolor=blue,citecolor=blue,urlcolor=blue,
  pdftitle={Sampled-Max Subgradient Method for Convex Finite-Max Optimization},
  pdfauthor={Egor Gladin, Anna Popova, Georgii Babinskii}}

\newcommand{\defeq}{\coloneqq}
\newcommand{\SMax}{\textsc{SMax}}
\newcommand{\X}{\mathcal X}

\newtheorem{theorem}{Theorem}
\newtheorem{lemma}{Lemma}
\newtheorem{proposition}{Proposition}
\newtheorem{corollary}{Corollary}
\newtheorem{remark}{Remark}
\newcounter{algorithm}

\title{Sampled-Max Subgradient Method for Convex Finite-Max Optimization}

\author{Egor Gladin \quad Anna Popova \quad Georgii Babinskii\\HSE University}
\date{}

\begin{document}

\maketitle

\begin{abstract}
We study the Sampled-Max Subgradient Method (\SMax--SGM) for large convex
finite-max problems.  Each iteration maximizes over a fresh random subset of
the $N$ components and takes one subgradient of the sampled maximizer.  The
method is therefore stochastic subgradient descent on a sampled-max surrogate.  We bound the surrogate error by an
average-top-$k$ gap plus the probability of missing all top-$k$ components.
A localization argument requires these quantities only on a near-optimal
sublevel set.  Under a bounded subgradient-moment assumption, this yields
$O(\varepsilon^{-2})$ subgradient queries and, when $k$ components remain
nearly active in that set,
$\widetilde O((N/k)\varepsilon^{-2})$ component-value queries, capped by the
full-scan cost.  Conversely, a one-dimensional affine construction shows that
$\Omega(N/k)$ component-value queries can be necessary even when the
average-top-$k$ gap vanishes everywhere and subgradient queries are unlimited.
A tensor-grid specialization
explains when the subset size can become independent of grid cardinality.
Experiments on finite maxima with $2\times10^5$ and $5\times10^6$ components
show that, under the reported protocols, \SMax--SGM reaches a 5\%
numerical-reference gap with fewer component-value queries and lower optimizer
time than each comparison method that reaches the target.
\end{abstract}

\section{Introduction}
\label{sec:introduction}

We study large-scale convex finite-max optimization: minimizing
\(F(x)=\max_{1\le i\le N}f_i(x)\) over \(x\in\X\), where the component
functions \(f_i\) are convex.
Such objectives arise when \(N\)
represents a large collection of data points, groups, tasks, or uncertainty
scenarios.  Although standard subgradient methods do not require smoothness,
computing a subgradient of \(F\) generally requires first identifying an
 active component, i.e., one attaining the maximum over \(i\) at a query point.
 Without additional information about the active index, this entails evaluating
 all \(N\) component values at every iteration, even
though only one component subgradient is subsequently used.  Consequently,
the cost of identifying the maximum can dominate the optimization procedure.

Finite maxima are a natural model for worst-case design.  They describe robust
optimization over finite scenarios \citep{BertsimasEtAl2011Robust}, maximal
loss over training examples \citep{ShalevShwartzWexler2016}, and empirical
group distributionally robust optimization, which minimizes the largest group
risk \citep{SagawaEtAl2020GroupDRO,YuEtAl2024GroupDRO}.  Discretized uniform
(Chebyshev) approximation also produces large finite maxima, including
classical minimax linear-phase finite-impulse-response (FIR) design
\citep{Rabiner1972FIRLP,McClellanParksRabiner1973} and variable
fractional-delay filter design \citep{ZhaoTay2023VFD}.  A different finite-max
application arises in private inference, where low-degree polynomial
activations replace nonlinearities that are expensive under homomorphic
encryption
\citep{GiladBachrachEtAl2016CryptoNets,LeeEtAl2023ApproxCNN,TongEtAl2024SmartPAF}.
Calibrating such replacements against downstream distortions over many
examples and classes produces another large finite maximum.  In the gridded
setting, finer grids improve fidelity but increase the cost of a full scan;
continuity simultaneously creates neighboring residuals with similar values,
suggesting that the maximum may be detectable from a small random subset.
Across these applications, the central question is whether one can avoid a
full scan without losing control of the original maximum objective.

We investigate this possibility through \SMax--SGM (Sampled-Max Subgradient Method).  At each iteration the
method samples \(m\) components uniformly without replacement, with \(m\ll N\)
in the favorable regime analyzed below, evaluates their values, selects the
sampled maximizer, and queries one stochastic subgradient of that component.
It therefore separates the number of component-value queries from the number
of stochastic subgradient queries.  The selected vector is an unbiased
stochastic subgradient of the convex surrogate
\(\widetilde F_m(x)=\mathbb E_S[\max_{i\in S}f_i(x)]\), so the method is
projected stochastic subgradient descent on \(\widetilde F_m\), rather than
directly on the full
maximum.

Our analysis identifies when optimizing the sampled surrogate still controls
the original maximum.  When many components are nearly maximal around the
relevant solutions, a small random subset is likely to contain a useful
component.  Localization requires this condition only in the region associated
with the target accuracy.  Under a bounded stochastic-subgradient moment
assumption, the usual \(O(\varepsilon^{-2})\) subgradient-query count is retained,
while the component-value query count can be sublinear in \(N\).  On
sufficiently fine Chebyshev grids, regularity creates
many similar neighboring residuals, and the required sample size can
eventually become independent of the grid cardinality.
The resulting value-query gain is structural and may diminish when only a
few components are nearly maximal.

This structure makes \SMax--SGM attractive when the goal is moderate accuracy
under a limited computational budget, because many components may be nearly
active and a small subset can give useful progress quickly.
Our experiments test this moderate-accuracy use case on a 200,000-component
variable fractional-delay filter design and a 5,000,000-component
polynomial-activation calibration problem motivated by private inference.  In
both cases, \SMax--SGM reaches a 5\% numerical-reference gap with substantially
fewer component-value queries than the competing methods that reach the
target, and it has the shortest optimizer time to that target.

Our contributions are:
\begin{itemize}
    \item We derive an average-top-\(k\) approximation bound for the sampled-max
    surrogate and a localization argument that transfers surrogate
    suboptimality to the original finite maximum.
    \item We combine this transfer bound with stochastic subgradient analysis
    to obtain a condition-dependent, sublinear-in-\(N\) value-query guarantee,
    while using only one stochastic subgradient per iteration.  An
    \(\Omega(N/k)\) value-query lower bound shows that its structural dependence
    on the number of near-active components cannot in general be improved.
    \item We specialize the theory to fine-grid Chebyshev approximation and
    evaluate the method on two large problems, documenting its early-budget
    performance and component-value savings.
\end{itemize}

\paragraph{Notation.}
For \(N\in\mathbb N\), let \([N]\defeq\{1,\ldots,N\}\).  We write
\(\|\cdot\|_2\) for the Euclidean norm, \(\partial f(x)\) for the convex
subdifferential of \(f\) at \(x\), and
\(\Pi_{\X}(x)\defeq\arg\min_{z\in\X}\|z-x\|_2\) for the Euclidean projection
onto a nonempty closed convex set \(\X\).  We denote the probability simplex by
\(\Delta_N\defeq\{\lambda\in\mathbb R_+^N:\sum_{i=1}^N\lambda_i=1\}\), and use
\(\mathrm{i}\defeq\sqrt{-1}\) for the imaginary unit.

\paragraph{Organization.}
Section~\ref{sec:related-work} reviews related finite-max methods.
Section~\ref{sec:method} presents \SMax--SGM and its localized guarantees,
and Section~\ref{sec:chebyshev-grid} specializes them to fine tensor grids.
Section~\ref{sec:experiments} reports the two numerical applications, and
Section~\ref{sec:conclusion} concludes.  All proofs and detailed experimental
protocols are collected in the appendix.

\section{Related Work}
\label{sec:related-work}

The direct baseline is projected subgradient descent on the original maximum.
It needs \(O(\varepsilon^{-2})\) iterations under standard Lipschitz and
bounded-domain assumptions, but identifying an active component costs \(N\)
values per iteration.  For the same nonsmooth finite-max class,
\citet{CarmonEtAl2021} use localized ball optimization to obtain
\(\widetilde O(N\varepsilon^{-2/3}+\varepsilon^{-8/3})\) component-value
evaluations and \(\widetilde O(\varepsilon^{-8/3})\) component-subgradient
evaluations.  They also prove an \(\Omega(N\varepsilon^{-2/3})\) lower bound in
their joint component first-order oracle, up to logarithmic factors.  This
worst-case lower bound is compatible with our localized result, which
identifies structured instances with sublinear value-query dependence on \(N\).

LogSumExp smooths the outer maximum; accelerated optimization of the smooth
surrogate takes \(\widetilde O(\varepsilon^{-1})\) full-gradient iterations,
corresponding to \(\widetilde O(N\varepsilon^{-1})\) component-value and
component-gradient evaluations in the standard dense implementation
\citep{Nesterov2005Smoothing,BeckTeboulle2012}.  Among convex overestimators of
the maximum that are 1-smooth with respect to \(\ell_\infty\), the
\(\Theta(\log N)\) worst-case approximation error of LogSumExp is optimal up to constants
\citep{SamakhoanaGrimmer2025}.  For finite maxima this acceleration requires
differentiable, sufficiently smooth component functions.  Adaptive smoothing
uses feedback precision adjustment \citep{PolakRoysetWomersley2003}.
\citet{PeeRoyset2011} combine exponential smoothing with an active-set strategy
and report favorable behavior when many functions are nearly active.  The
exact gradient of the standard LogSumExp still aggregates all
\(N\) components.  Smooth active-component identification can reduce the
working set under additional differentiability assumptions
\citep{RasTamUteda2025}, while recent work proposes a stochastic approximation
of LogSumExp that avoids the full sum by optimizing a different smooth
surrogate \citep{GladinEtAl2025LSE}.

The identity
\(
    \max_{i\in[N]} f_i(x)
    =\max_{\lambda\in\Delta_N}\sum_{i=1}^N\lambda_i f_i(x)
\)
gives a convex--concave saddle formulation with a simplex-valued dual variable.
For smooth components, Mirror-Prox and general nonlinear-coupling primal--dual
methods apply \citep{Nemirovski2004MirrorProx,HamedaniAybat2021}; nonsmooth and
finite-sum variants are treated by \citet{ZhuLiuTranDinh2022} and
\citet{AlacaogluMalitsky2022}.  \citet{RasTamUteda2025} study this formulation
specifically for smooth finite-max minimization and active-component
identification; empirical group DRO is an application-specific example
\citep{YuEtAl2024GroupDRO}.  In the finite-max specialization, these methods
maintain an \(N\)-dimensional dual vector and impose operator assumptions that
differ from our component-oracle model.
Working-set and exchange methods instead maintain a selected component or
constraint set.
They range from incremental methods using incomplete objective information
\citep{GaudiosoEtAl2006} to SQP and exchange algorithms designed for finely
discretized minimax and semi-infinite programs
\citep{ZhouTits1996,ZhangEtAl2010Exchange}.

A direct algorithmic predecessor of \SMax--SGM is Ordered SGD
\citep{KawaguchiLu2020}, which samples a minibatch and differentiates the
average of its \(q\) largest losses.  When \(q=1\), its update coincides with
\SMax--SGM and its ordered objective coincides with \(\widetilde F_m\) (up to a separate
regularizer).  Ordered SGD establishes convergence to this ordered objective;
our contribution is an explicit localized guarantee relating the sampled
surrogate to the original finite maximum.  Average-top-\(k\) loss similarly
targets a different aggregation \citep{FanEtAl2017TopK}.

Table~\ref{tab:related-work} summarizes the closest first-order query bounds.
We separate component-value from component-(sub)gradient evaluations whenever
the analysis does so.

\FloatBarrier
\begin{table}[!htb]
\caption{Representative component-query upper and lower bounds for convex
finite-max optimization.  A check mark indicates support for nonsmooth
components.  The bounds suppress geometry and moment constants, and
\(\widetilde O\) hides logarithmic factors.}
\label{tab:related-work}
\centering
\small
\setlength{\tabcolsep}{3pt}
\renewcommand{\arraystretch}{1.08}
\begin{tabular}{>{\raggedright\arraybackslash}p{0.24\linewidth}
                >{\centering\arraybackslash}m{0.12\linewidth}
                >{\centering\arraybackslash}m{0.27\linewidth}
                >{\centering\arraybackslash}m{0.27\linewidth}}
\toprule
Approach & Nonsmooth? & Value queries & (Sub)gradient queries \\
\midrule
Full-scan SGM \citep{NemirovskiEtAl2009SA,CarmonEtAl2021}
& \(\checkmark\)
& \(O(N\varepsilon^{-2})\)
& \(O(\varepsilon^{-2})\) \\

Smooth surrogate \citep{Nesterov2005Smoothing,BeckTeboulle2012}
& \(\times\)
& \(\widetilde O(N\varepsilon^{-1})\)
& \(\widetilde O(N\varepsilon^{-1})\) \\

Ball-oracle acceleration \citep{CarmonEtAl2021}
& \(\checkmark\)
& \(\widetilde O(N\varepsilon^{-2/3}+\varepsilon^{-8/3})\)
& \(\widetilde O(\varepsilon^{-8/3})\) \\

\textbf{\SMax--SGM} (Theorem~\ref{thm:localized-complexity})
& \(\checkmark\)
& \(\widetilde O\!\left(\frac Nk\varepsilon^{-2}\right)\)
& \(O(\varepsilon^{-2})\) \\
\midrule
Joint-oracle lower bound \citep{CarmonEtAl2021}
& \(\checkmark\)
& \multicolumn{2}{c}{\(\Omega(N\varepsilon^{-2/3})\) joint queries} \\

Our value lower bound (Theorem~\ref{thm:value-lower-bound})
& \(\checkmark\)
& \(\Omega(N/k)\)
& unrestricted \\
\bottomrule
\end{tabular}
\end{table}

The smooth-surrogate rate requires Lipschitz gradients.  The \SMax--SGM rate
assumes the localized average-top-\(k\) condition and a uniform component-moment
bound; its exact capped rate is~\eqref{eq:query-complexity}.  The lower-bound
rows are incomparable: \citet{CarmonEtAl2021} count joint value-and-first-order
queries in the worst case, with no \(k\) parameter, whereas
Theorem~\ref{thm:value-lower-bound} assumes \(\operatorname{gap}_k(\X)=0\) but
forces value queries even when subgradient queries are unrestricted.

\section{\SMax--SGM and Localized Guarantees}
\label{sec:method}

We consider
\begin{equation}
    \min_{x\in\X} F(x),
    \qquad
    F(x)\defeq\max_{i\in[N]}f_i(x),
    \label{eq:problem}
\end{equation}
where \(\X\subseteq\mathbb R^d\) is nonempty, compact, and convex, and every
\(f_i:\mathbb R^d\to\mathbb R\) is convex.  Let
\(D_{\X}\defeq\sup_{x,z\in\X}\|x-z\|_2\).  Component-value queries return
\(f_i(x)\) exactly.  We assume that, conditional on the information
\(\mathcal H\) available before an oracle call, the stochastic subgradient
returned at any \(\mathcal H\)-measurable query \((x,i)\) satisfies, almost
surely,
\[
    \mathbb E[g_i(x;\zeta)\mid\mathcal H]\in\partial f_i(x),
    \qquad
    \mathbb E[\|g_i(x;\zeta)\|_2^2\mid\mathcal H]\le G^2.
\]

\begin{center}
\refstepcounter{algorithm}\label{alg:smax}
\fbox{\begin{minipage}{0.94\linewidth}
\textbf{Algorithm \thealgorithm: \SMax--SGM}\\[-0.25ex]
\textbf{Input:} horizon \(T\in\mathbb N\), initial point \(x_1\in\X\), subset
size \(m\in[N]\),
stepsizes \(\eta_1,\ldots,\eta_T>0\), and a deterministic tie-breaking rule.
\smallskip

\begin{tabular}{@{}r@{\hspace{0.55em}}p{0.84\linewidth}@{}}
1: & \textbf{for} \(t=1,\ldots,T\) \textbf{do} \\
2: & \quad Draw a uniform size-\(m\) subset \(S_t\subset[N]\). \\
3: & \quad Query \(f_i(x_t)\) for \(i\in S_t\), and choose
      \(I_t\in\arg\max_{i\in S_t}f_i(x_t)\). \\
4: & \quad Query \(g_{I_t}(x_t;\zeta_t)\), and set
      \(x_{t+1}=\Pi_{\X}(x_t-\eta_tg_{I_t}(x_t;\zeta_t))\). \\
5: & \textbf{end for} \\
6: & \textbf{return} \(\bar x_T=T^{-1}\sum_{t=1}^T x_t\). \\
\end{tabular}
\end{minipage}}
\end{center}

Each iteration uses \(m\) component-value queries and one stochastic
subgradient query.  To identify its objective, define the convex surrogate
\begin{equation}
    \widetilde F_m(x)
    \defeq
    \mathbb E_S[F_S(x)],
    \qquad
    F_S(x)\defeq\max_{i\in S}f_i(x),
    \label{eq:smax-surrogate}
\end{equation}
where \(S\) is uniform among all size-\(m\) subsets of \([N]\).  Let
\(\mathcal H_t\) be the history before \(S_t\) is drawn.

\begin{lemma}
\label{lem:surrogate-gradient}
For every iteration of Algorithm~\ref{alg:smax},
\[
    \mathbb E[g_{I_t}(x_t;\zeta_t)\mid\mathcal H_t]
    \in\partial\widetilde F_m(x_t),
    \qquad
    \mathbb E[\|g_{I_t}(x_t;\zeta_t)\|_2^2\mid\mathcal H_t]\le G^2.
\]
Consequently, Algorithm~\ref{alg:smax} is projected stochastic subgradient
descent on \(\widetilde F_m\), using the selected component's subgradient.
\end{lemma}

The proof is in Appendix~\ref{app:surrogate-gradient}.  The appendix also
replaces the uniform scale \(G\) by a sharper selection-weighted scale
\(G_m\le G\), which averages component moment bounds according to their
probabilities of being selected by the sampled maximum.

\subsection{Approximation and Localized Complexity}
\label{sec:theory}

Write \(f_{(1)}(x)\ge\cdots\ge f_{(N)}(x)\) for the sorted component values and
define the average-top-\(k\) objective
\[
    \operatorname{AT}_k(x)\defeq\frac1k\sum_{j=1}^k f_{(j)}(x).
\]
For a region \(\mathcal R\subseteq\X\), let
\begin{align}
    \operatorname{gap}_k(\mathcal R)
    &\defeq\sup_{x\in\mathcal R}
      \bigl[F(x)-\operatorname{AT}_k(x)\bigr],
    &
    W(\mathcal R)
    &\defeq\sup_{x\in\mathcal R}
      \bigl[f_{(1)}(x)-f_{(N)}(x)\bigr].
    \label{eq:gap-range}
\end{align}
The quantity \(\mathrm{gap}_k\) measures how far the average of the largest \(k\) components can fall below the maximum, while \(W\) is the full component-value width.

\begin{proposition}
\label{prop:approximation}
For every \(x\in\X\) and \(k\in[N]\), the sampled surrogate is a lower
approximation to the maximum and satisfies
\begin{equation}
    0\le F(x)-\widetilde F_m(x)
    \le F(x)-\operatorname{AT}_k(x)
    +\bigl[f_{(1)}(x)-f_{(N)}(x)\bigr]\pi^{\rm miss}_{m,k},
    \qquad
    \pi^{\rm miss}_{m,k}\defeq
      \frac{\binom{N-k}{m}}{\binom{N}{m}}.
    \label{eq:pointwise-approximation}
\end{equation}
Here \(\binom{N-k}{m}=0\) when \(m>N-k\).  Moreover,
\(\pi^{\rm miss}_{m,k}\le\exp(-mk/N)\), and hence, for every
\(\mathcal R\subseteq\X\),
\begin{equation}
    \sup_{x\in\mathcal R}[F(x)-\widetilde F_m(x)]
    \le \operatorname{gap}_k(\mathcal R)+W(\mathcal R)e^{-mk/N}.
    \label{eq:regional-approximation}
\end{equation}
\end{proposition}

The proof is in Appendix~\ref{app:approximation-proof}.  The probability
\(\pi^{\rm miss}_{m,k}\) is exactly that of missing all top-\(k\) components.
Thus the approximation improves when many component values are close to the
maximum, even if \(m\ll N\).

Denote
\[
	F^\star\defeq\min_{x\in\X}F(x),
	\qquad
	\widetilde F_m^\star\defeq\min_{x\in\X}\widetilde F_m(x).
\]
The moment assumption also ensures that the minima are attained.  Indeed,
\(s_i(x)\defeq\mathbb E[g_i(x;\zeta)]\in\partial f_i(x)\) and Jensen's
inequality give \(\|s_i(x)\|_2\le G\).  Applying the subgradient inequality at
two points shows that every \(f_i\) is \(G\)-Lipschitz on \(\X\), and so are \(F\) and
\(\widetilde F_m\).  Since \(\X\) is compact,
both objectives attain their minima.

Crucially, the discussed approximation need not hold uniformly on all of \(\X\).  Let
\[
    \X_\rho\defeq\{x\in\X:F(x)\le F^\star+\rho\}.
\]
To state the localized transfer, define the localized approximation error
\begin{equation}
    \operatorname{err}_{m,k}(\rho)
    \defeq
    \operatorname{gap}_k(\X_\rho)
    +W(\X_\rho)\pi^{\rm miss}_{m,k},
    \label{eq:localized-error}
\end{equation}
and surrogate suboptimality 
\[
    \operatorname{subopt}_m(x)
    \defeq\widetilde F_m(x)-\widetilde F_m^\star.
\]

\begin{theorem}
\label{thm:localized-transfer}
Let \(\widehat x\) be any random point in \(\X\) satisfying
\(\mathbb E[\operatorname{subopt}_m(\widehat x)]\le\beta\).  If
\(\operatorname{err}_{m,k}(\rho)<\rho\), then
\begin{equation}
    \mathbb P\!\left(F(\widehat x)-F^\star\le\rho\right)
    \ge 1-\frac{\beta}{\rho-\operatorname{err}_{m,k}(\rho)}.
    \label{eq:localized-transfer-probability}
\end{equation}
Moreover,
\begin{equation}
    \mathbb E[F(\widehat x)-F^\star]
    \le \operatorname{err}_{m,k}(\rho)+\beta+GD_{\X}
    \min\!\left\{1,
      \frac{\beta}{\rho-\operatorname{err}_{m,k}(\rho)}\right\}.
    \label{eq:localized-transfer-expectation}
\end{equation}
\end{theorem}

The proof is in Appendix~\ref{app:localization}.  Its key step is a sublevel
barrier.  If \(x\notin\X_\rho\), convexity along the segment from a minimizer
of \(F\) to \(x\), together with the surrogate bound on the boundary of
\(\X_\rho\), gives
\(\operatorname{subopt}_m(x)>\rho-\operatorname{err}_{m,k}(\rho)\).
Markov's inequality then yields~\eqref{eq:localized-transfer-probability}.
Thus the approximation need only hold where the target accuracy is relevant.

We next combine this transfer with stochastic subgradient convergence to
obtain a localized query-complexity bound.

\begin{theorem}
\label{thm:localized-complexity}
For a target \(\varepsilon>0\) and confidence parameter \(\delta\in(0,1)\),
suppose that some \(k\in[N]\) satisfies
\begin{equation}
    \operatorname{gap}_k(\X_\varepsilon)\le\frac{\varepsilon}{4}.
    \label{eq:localized-topk}
\end{equation}
Choose
\begin{equation}
    m=\min\!\left\{N-k+1,\;
        \max\!\left\{1,
        \left\lceil\frac Nk
        \log\!\left(1+\frac{4W(\X_\varepsilon)}{\varepsilon}\right)
        \right\rceil\right\}\right\}.
    \label{eq:subset-choice}
\end{equation}
Assume \(D_{\X}G>0\) and set
\begin{equation}
	T=
	\left\lceil\frac{4D_{\X}^2G^2}
	{\delta^2\varepsilon^2}\right\rceil,
	\label{eq:iteration-choice}
\end{equation}
and run Algorithm~\ref{alg:smax} with
\(\eta_t\equiv D_{\X}/(G\sqrt T)\).  Its output satisfies
\[
    \mathbb P\bigl(F(\bar x_T)-F^\star\le\varepsilon\bigr)\ge1-\delta.
\]
Consequently, at fixed confidence,
the numbers of stochastic-subgradient and component-value oracle calls used by
the method satisfy
\begin{align}
    \mathsf Q_{\rm grad}&=O(1+D_{\X}^2G^2/\varepsilon^2),
    \nonumber\\
    \mathsf Q_{\rm val}&=O\!\left(
       \left[1+\frac{D_{\X}^2G^2}{\varepsilon^2}\right]
       \min\!\left\{N-k+1,\;
       1+\frac Nk\log\!\left(1+\frac{W(\X_\varepsilon)}{\varepsilon}\right)
       \right\}\right).
    \label{eq:query-complexity}
\end{align}
\end{theorem}

The proof is in Appendix~\ref{app:complexity}, which first establishes the
corresponding complexity guarantee with the selection-weighted scale \(G_m\)
in place of \(G\); the main-text result then follows from \(G_m\le G\).  If
\(D_{\X}G=0\), every point of \(\X\) is optimal and the conclusion is
immediate.  This is an
instance-dependent guarantee, not an a priori tuning rule: \(F^\star\),
\(\X_\varepsilon\), and the localized gap and width are generally unknown.
Analytic upper bounds can determine \(m\); our experiments instead tune it on
training-only development runs.  With geometry and moment scales fixed, if
\(k\) is of order \(N^\alpha\) for some \(\alpha\in(0,1]\), and
\(W(\X_\varepsilon)/\varepsilon\) is at most polynomial, then the
value-query dependence is \(\widetilde O(N^{1-\alpha})\), rather than the
\(O(N)\) cost of a full scan per iteration; when \(k=O(1)\), the bound
recovers \(O(N)\).

The factor \(N/k\) reflects a genuine information requirement in the separated
value/subgradient oracle model, rather than only the analysis of \SMax--SGM.

\begin{theorem}
\label{thm:value-lower-bound}
Let \(N\) be even, \(1\le k\le N/4\), and \(R,G>0\).  There is a family of
one-dimensional affine instances of~\eqref{eq:problem} on
\(\X=[-R,R]\) with deterministic component-subgradient oracles such that every
instance satisfies
\[
    \operatorname{gap}_k(\X)=0,
    \qquad
    W(\X)\le 6GR,
\]
and all component subgradients have norm \(G\).  For every
\(0<\varepsilon\le GR/2\), any possibly randomized algorithm that is
\(\varepsilon\)-accurate with probability at least \(2/3\) on every instance in
the family and uses at most \(\mathsf Q_{\rm val}\) component-value queries must
satisfy
\begin{equation}
    \mathsf Q_{\rm val}\ge \frac{N}{6k}.
    \label{eq:value-lower-bound}
\end{equation}
This holds regardless of the number of component-subgradient queries.
\end{theorem}

The proof is in Appendix~\ref{app:value-lower-bound}.  It uses instances with
the stronger global property \(F=\operatorname{AT}_k\), so it applies inside
the localized class of Theorem~\ref{thm:localized-complexity}.
Thus, the \(N/k\) dependence is unavoidable in general, although this lower
bound does not address the \(\varepsilon\)- or logarithmic dependence
in~\eqref{eq:query-complexity}.

\section{Fine Tensor-Grid Chebyshev Approximation}
\label{sec:chebyshev-grid}

The near-active condition has a simple interpretation for Chebyshev
(uniform-norm) approximation on a tensor grid.  For \(q\ge1\) and integers
\(n_\ell\ge2\), introduce the discretization
\[
    \mathcal T=\prod_{\ell=1}^q
      \left\{0,\frac1{n_\ell-1},\ldots,1\right\},
    \qquad N=\prod_{\ell=1}^q n_\ell,
\]
and let
\(\tau_{\boldsymbol i}\in\mathcal T\) be a grid node.  For coefficients
\(x\in\X\), let \(r_x(\tau)\) be the pointwise approximation error at
\(\tau\in [0,1]^q\), assumed convex in \(x\).  The finite-max components are
\(f_{\boldsymbol i}(x)=r_x(\tau_{\boldsymbol i})\), and hence
\(F(x)=\max_{\tau\in\mathcal T}r_x(\tau)\).  For example, if \(b\) is the
target function and \(\phi(\tau)\) is the vector of values of prescribed basis
functions at \(\tau\), then \(x\) contains their coefficients,
\(\langle\phi(\tau),x\rangle\) is the approximant, and
\[
    r_x(\tau)=|\langle\phi(\tau),x\rangle-b(\tau)|.
\]
Suppose there exists a constant \(L_\tau\ge0\) such that, uniformly over
\(x\in\X\) and for all \(\tau,\tau'\in[0,1]^q\),
\begin{equation}
    |r_x(\tau)-r_x(\tau')|
    \le L_\tau\|\tau-\tau'\|_\infty.
    \label{eq:chebyshev-components}
\end{equation}
For the linear approximation model above, the condition holds with
\(L_\tau=B_{\X}L_\phi+L_b\) when
\(\sup_{x\in\X}\|x\|_2\le B_{\X}\),
\(\|\phi(\tau)-\phi(\tau')\|_2\le
L_\phi\|\tau-\tau'\|_\infty\), and
\(|b(\tau)-b(\tau')|\le L_b\|\tau-\tau'\|_\infty\).

The resulting tensor-grid specialization is as follows.

\begin{corollary}
\label{cor:fine-grid}
Under~\eqref{eq:chebyshev-components}, let \(\varepsilon>0\) and define
\(A_\varepsilon\defeq 1+4L_\tau/\varepsilon\).  There exists
\(k_\varepsilon\in[N]\) such that
\[
    \operatorname{gap}_{k_\varepsilon}(\X)\le\frac{\varepsilon}{4},
    \qquad W(\X)\le L_\tau,
    \qquad \frac{N}{k_\varepsilon}\le A_\varepsilon^q.
\]
Consequently, Theorem~\ref{thm:localized-complexity} permits a subset size
\[
    m\le \min\!\left\{N,\,
      1+A_\varepsilon^q\log A_\varepsilon\right\}.
\]
If component subgradients have norm at most \(G_r\), then, at fixed
confidence, the value-query bound is
\begin{equation}
    O\!\left(
      \left[1+\frac{D_{\X}^2G_r^2}{\varepsilon^2}\right]
      \min\!\left\{N,\,1+A_\varepsilon^q\log A_\varepsilon\right\}
      \right),
    \label{eq:grid-complexity}
\end{equation}
while a full-grid subgradient method uses
\(O(N[1+D_{\X}^2G_r^2/\varepsilon^2])\) component values.
\end{corollary}

Appendix~\ref{app:grid-details} gives the explicit construction and proof.  It
also controls the discretization error: for
\(H(x)=\sup_{\tau\in[0,1]^q}r_x(\tau)\),
\[
    0\le H(x)-F(x)\le
    \Delta_{\mathcal T}\defeq\frac{L_\tau}{2}
      \max_\ell\frac1{n_\ell-1}.
\]
Thus every \(\xi\)-suboptimal point for \(F\) is
\((\xi+\Delta_{\mathcal T})\)-suboptimal for \(H\).  Refining a fixed-dimensional tensor
grid therefore improves continuous-domain fidelity while creating more nearly
active discrete components.  The case \(q=2\) describes the grid-cardinality
scaling of the VFD experiment below after rescaling its rectangle to
\([0,1]^2\).

\FloatBarrier
\section{Numerical Experiments}
\label{sec:experiments}

We evaluate \SMax--SGM on two large finite-max instances.
The experimental code is available at
\url{https://github.com/egorgladin/smax-sgm}.
We use component-value queries as our primary efficiency metric,
while runtime illustrates additional computational costs, particularly
for Restricted SOCP. Runtime comparisons should be interpreted with
caution, as they depend on the implementation, computing environment,
and timing protocol.

\subsection{Variable Fractional-Delay Design}
\label{sec:vfd-experiment}

Variable fractional-delay (VFD) design arises in digital signal processing
when a realizable filter must approximate a delay that varies continuously by
a fraction of the sampling period.  A common implementation is the Farrow
structure, whose fixed subfilters are combined polynomially in an adjustable
delay parameter \(p\) \citep{ZhaoTay2023VFD}.  We adopt the
algebraic-polynomial VFD model and grid from that work, but replace its weighted
least-squares criterion by a minimax criterion to control the largest complex
response error jointly over frequency and delay:
\begin{equation}
    H_a(\omega,p)
    =\sum_{\nu=0}^{L_h}\sum_{s=0}^{r}
      a(\nu,s)p^s e^{-\mathrm{i}\omega(\nu-L_h/2)},
    \qquad
    D(\omega,p)=e^{-\mathrm{i}\omega p}.
    \label{eq:vfd-response}
\end{equation}
Here \(H_a\) is the implemented response and \(D\) is the ideal response;
\(L_h\) and \(r\) are the filter and polynomial orders.
On a tensor grid, each component is the complex-modulus error
\(f_{i\ell}(a)=|H_a(\omega_i,p_\ell)-D(\omega_i,p_\ell)|\), a convex norm of
an affine function.  We take \(L_h=60\), \(r=4\), 1,000 frequencies in
\([0,0.9\pi]\), and 200 delays in \([0,0.5]\).  After symmetry reduction the
problem has \(d=153\) real variables and \(N=200{,}000\) components.  A denser
\(2001\times401\) validation grid checks the final response.

Every method starts from the same tensor least-squares point.  We compare
\SMax--SGM with full-grid SGM, LSE continuation with L-BFGS-B, and
restricted-SOCP exchange.  The last method supplies a tight numerical reference
interval, reported in Appendix~\ref{app:vfd-protocol}; its lower endpoint
defines the relative training gap.  We omit the method of
\citet{CarmonEtAl2021}: its ball-optimization-oracle construction is
substantially more involved to implement, and the paper reports no experiments
to guide a faithful baseline.  Every method is assessed under the common
budget \(25N=5\times10^6\) component values.  Tunable parameters are selected
by the final training objective at this same horizon, without using the 5\%
crossing time.  \SMax\ uses \(m=16384\), selected on development seeds 0--2
and then frozen for 20 disjoint confirmatory seeds.  Validation values are not
used for tuning, stopping, or output selection, and all tuning cost is excluded.
Restricted SOCP is continued to numerical convergence only to obtain the
common numerical reference.

\FloatBarrier
\begin{table}[H]
    \caption{VFD optimizer-only cost to the first stored 5\% checkpoint and
    endpoint performance.  For \SMax, target costs are medians of seed-specific
    first crossings and endpoints are medians over seeds; ``Val.'' is the
    denser-grid maximum error.  Times are warmed.}
    \label{tab:vfd-results}
    \centering
    \footnotesize
    \setlength{\tabcolsep}{3.0pt}
    \begin{tabular}{@{}l*{4}{c}@{}}
        \toprule
        Method & Value queries to 5\% & Time to 5\% (s) & Final gap & Val. \\
        \midrule
        \SMax--SGM & \(2.61\!\times\!10^6\) & 0.545 & 1.82\% & \(2.75\!\times\!10^{-3}\) \\
        Full-grid SGM & -- & -- & 27.17\% & \(3.44\!\times\!10^{-3}\) \\
        LSE--L-BFGS & -- & -- & 44.99\% & \(3.92\!\times\!10^{-3}\) \\
        Restricted SOCP & \(3.40\!\times\!10^6\) & 39.3 & 3.12\% & \(2.79\!\times\!10^{-3}\) \\
        \bottomrule
    \end{tabular}
\end{table}

All 20 \SMax\ runs reach 5\%; the query and warmed-runtime IQRs are
2,211,840--2,605,056 and 0.479--0.596 seconds, respectively.  Relative to
SOCP, the median \SMax\ crossing uses 1.31 times fewer values and 72 times less
warmed optimizer time; full-grid SGM and LSE do not cross within \(25N\).
Timings exclude the common WLS initialization, tuning, scoring, and one-time
setup; the excluded \SMax\ compilation took 1.49 seconds.

\begin{figure}[H]
    \centering
    \includegraphics[width=\linewidth]
    {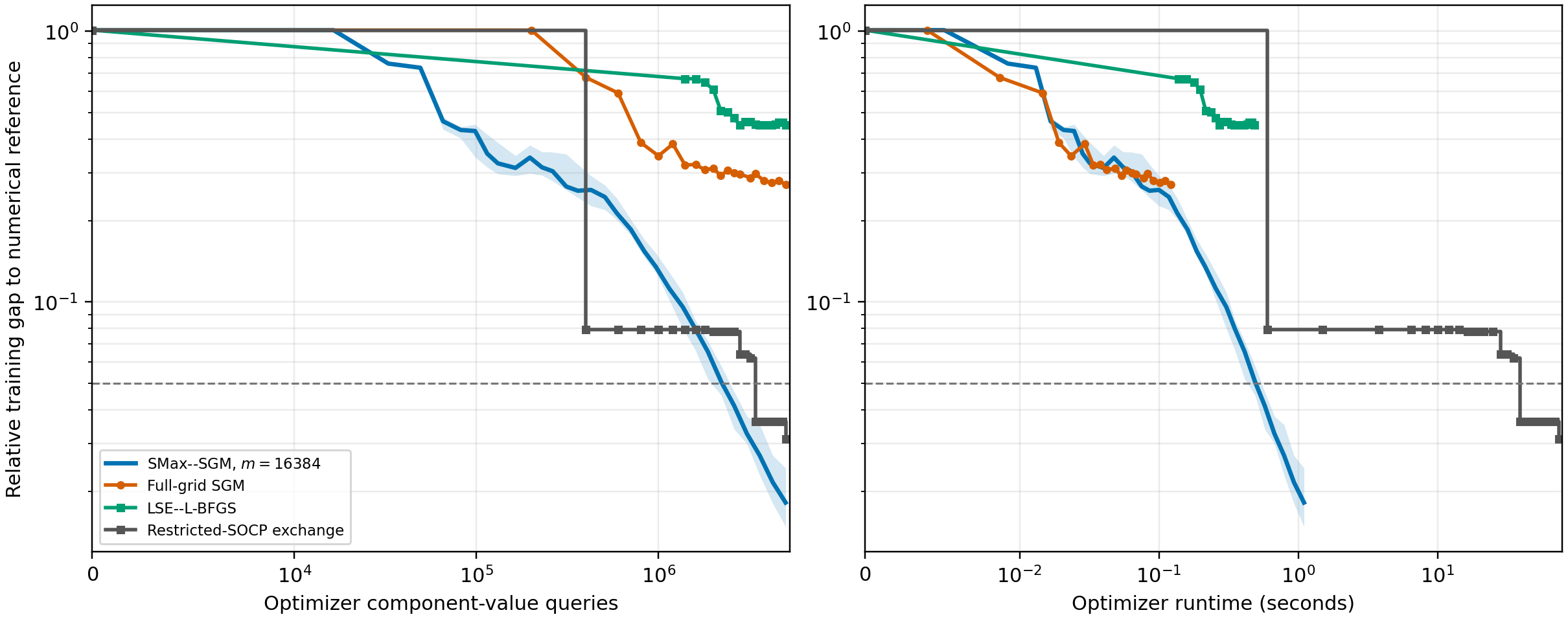}
    \caption{VFD relative training gap to the numerical SOCP reference versus
    post-WLS value queries (left) and warmed optimizer runtime
    (right).  \SMax\ shows the median and interquartile range over 20
    confirmatory seeds.  Curves are pointwise summaries; Table~\ref{tab:vfd-results}
    instead uses seed-specific first crossings.  The dashed line marks 5\%; axes
    exclude initialization, tuning, setup, and scoring.}
    \label{fig:vfd-results}
\end{figure}

\FloatBarrier
\subsection{Downstream-Logit Polynomial Calibration}
\label{sec:experiment}

Our second experiment is a data-dependent finite maximum arising when the
final post-addition ReLU module in the public CIFAR-100
\texttt{cifar100\_resnet20} checkpoint
from \url{https://github.com/chenyaofo/pytorch-cifar-models}
\citep{Krizhevsky2009CIFAR}, which implements the
residual architecture of \citet{HeEtAl2016ResNet}, is replaced by a degree-14
polynomial.  Polynomial activations are
useful in private inference because additions and multiplications are much
cheaper than general nonlinearities under homomorphic encryption
\citep{LeeEtAl2023ApproxCNN,TongEtAl2024SmartPAF}.  This experiment performs
calibration and evaluation in plaintext; it does not benchmark encrypted
execution.  Only global-average pooling and a linear classifier follow this
module, so every logit change is affine in
the 15 polynomial coefficients.  We minimize the largest centered logit
distortion over all 50,000 training images and 100 classes.  Centering removes
each image's common class-logit shift, which does not change its softmax
probabilities.  This yields \(N=5{,}000{,}000\) absolute affine components in only
15 coefficients.  The feasible box is anchored at a scalar near-minimax ReLU
polynomial and guarantees scalar approximation error at most 3.584 over the
training activation range; Appendix~\ref{app:experiment} gives the construction.

All iterative methods start from the same bounded weighted-least-squares fit.
We compare \SMax--SGM with full-grid SGM, log-sum-exp (LSE) continuation with
L-BFGS-B, and restricted LP exchange.  Each general method receives a budget
of \(25N\) component-value queries; LP exchange stops after two exact separation
scans.  All tunable configurations are selected by their endpoint training
maximum at this same \(25N\) horizon, without optimizing the 5\% crossing
time.  The \SMax\ subset size is \(m=8192\), selected on development seeds
0--2, then frozen and evaluated over 20 disjoint confirmatory seeds.  All
tuning cost is excluded, and no test metric is used for tuning, stopping, or
output selection.  We report the relative training gap to the numerical LP
reference, post-WLS value queries, warmed optimizer runtime, and held-out
performance.

\FloatBarrier
\begin{table}[!htb]
    \caption{Optimizer-only cost to the first stored 5\% checkpoint and
    endpoint performance.  For \SMax, target costs are medians of seed-specific
    first crossings and endpoints are medians over seeds.  Times are warmed; a
    dash means the method did not reach the target.}
    \label{tab:polyact-results}
    \centering
    \footnotesize
    \setlength{\tabcolsep}{3.0pt}
    \begin{tabular}{@{}l*{5}{c}@{}}
        \toprule
        Method & Value queries to 5\% & Time to 5\% (s) & Final gap & Test max & Acc. \\
        \midrule
        \SMax--SGM & \(2.65\!\times\!10^6\) & 0.875 & 1.30\% & 1.67 & 68.56\% \\
        Full-grid SGM & -- & -- & 8.25\% & 1.67 & 68.55\% \\
        LSE--L-BFGS-B & \(9.0\!\times\!10^7\) & 15.5 & 3.25\% & 1.64 & 68.60\% \\
        Restricted LP & \(1.0\!\times\!10^7\) & 1.47 & \(<10^{-4}\%\) & 1.69 & 68.55\% \\
        \bottomrule
    \end{tabular}
\end{table}

All 20 \SMax\ runs reach 5\%; the query and warmed-runtime IQRs are
2,015,232--3,766,272 and 0.664--1.268 seconds, respectively, with a median of
324 selected subgradients.  Relative to LP and LSE, the median crossing uses
3.77 and 33.9 times fewer values, with warmed-runtime speedups of 1.68 and
17.7.  Including the common streamed WLS fit gives about \(1.53N\) values;
full-grid SGM does not cross within \(25N\).  The original checkpoint has
68.83\% accuracy, and \SMax\ predictions agree with it on 97.21\% of test
images.

We also compute the pointwise effective near-active count
\(\widehat k_\varepsilon(x)=\max\{k:F_{\rm train}(x)-
\operatorname{AT}_k(x)\le\varepsilon/4\}\), with \(\varepsilon=0.05\ell\).
At the first 5\% crossing its median is 85.5 for VFD but 1 for polynomial
calibration; detailed stagewise results and the pointwise-versus-regional
qualification appear in Appendix~\ref{app:near-active-diagnostic}.

\begin{figure}[H]
    \centering
    \includegraphics[width=\linewidth]
    {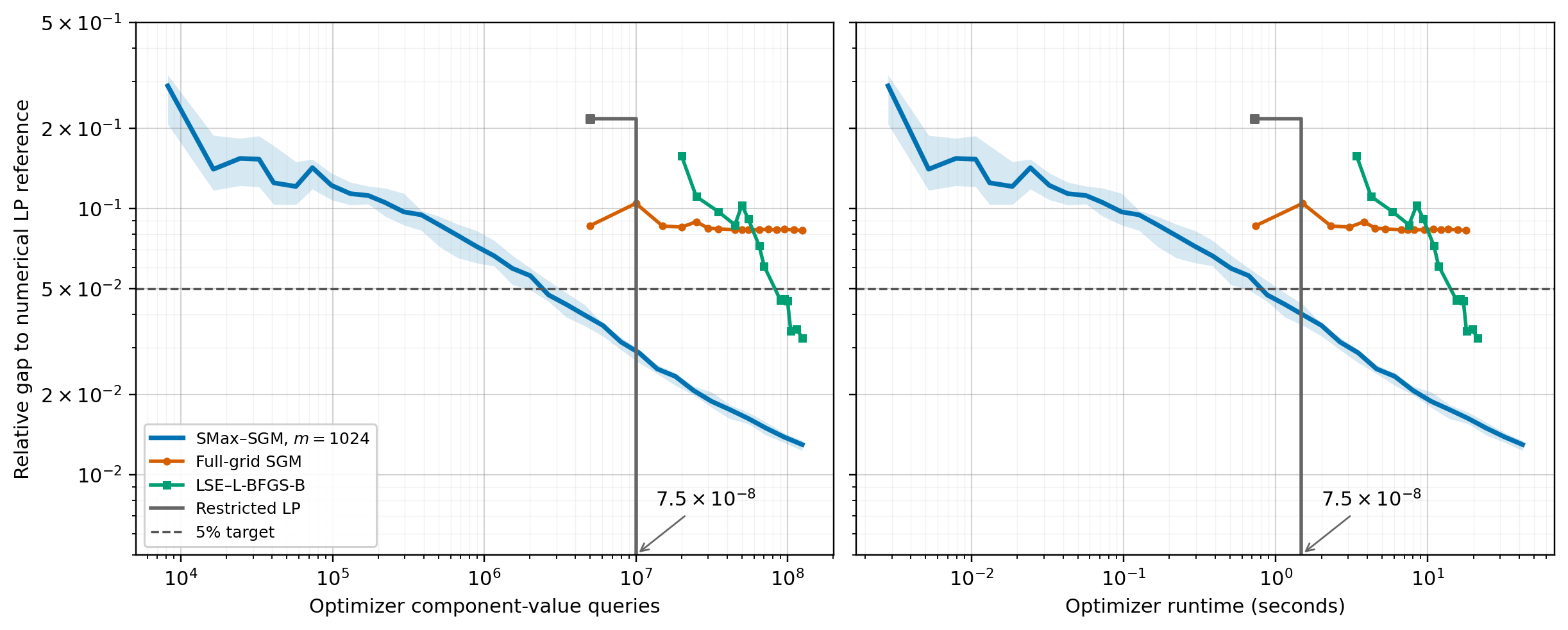}
    \caption{Relative training gap to the numerical LP reference versus
    post-WLS value queries (left) and warmed optimizer runtime
    (right).  \SMax\ shows the median and interquartile range over 20
    confirmatory seeds.  Curves are pointwise summaries; Table~\ref{tab:polyact-results}
    instead uses seed-specific first crossings.  The dashed line marks 5\%; axes
    exclude the common warm start, tuning, setup, and offline scoring.}
    \label{fig:polyact-results}
\end{figure}

\FloatBarrier
\section{Conclusion}
\label{sec:conclusion}

We analyzed \SMax--SGM as stochastic subgradient descent on a sampled-maximum
surrogate and derived a localized condition for sublinear value queries.  The
tensor-grid corollary explains why, at fixed accuracy, refining a
fixed-dimensional grid can eventually leave the required subset size
independent of the number of grid nodes.
The value-query lower bound shows, however, that without additional structure
the \(N/k\) dependence itself cannot be removed, even under perfect top-\(k\)
activity and with unrestricted subgradient queries, although it does not
address the \(\varepsilon\)- or logarithmic dependence of our upper bound.

In experiments with 200,000 and five million components, \SMax--SGM reached a
5\% numerical-reference gap with fewer
component-value queries and lower optimizer time than every comparison method
that reached the target under the reported protocols.  The gain remains
instance dependent and may disappear when few components are nearly maximal.
The pointwise \(\widehat k_\varepsilon\) measurements support this near-active
mechanism in VFD but remain small in polynomial calibration, emphasizing that
the theoretical condition is sufficient rather than necessary for an
empirical gain.
Together, the theory and experiments support sampled maxima as a simple way to
avoid repeated full scans when moderate accuracy is sufficient and near-active
components are sufficiently abundant.

\setlength{\bibsep}{0pt}
\bibliography{references}

@book{hardy1952inequalities,
  title={Inequalities},
  author={Hardy, Godfrey Harold and Littlewood, John Edensor and P{\'o}lya, George},
  edition={Second},
  year={1952},
  publisher={Cambridge University Press}
}

@article{BarrodalePhillips1975,
  author  = {Barrodale, Ian and Phillips, Curtis},
  title   = {Algorithm 495: Solution of an Overdetermined System of Linear Equations in the {Chebyshev} Norm},
  journal = {ACM Transactions on Mathematical Software},
  volume  = {1},
  number  = {3},
  pages   = {264--270},
  year    = {1975},
  doi     = {10.1145/355644.355651}
}

@article{GoulartChen2026Clarabel,
  author  = {Goulart, Paul J. and Chen, Yuwen},
  title   = {Clarabel: An Interior-Point Solver for Conic Programs with Quadratic Objectives},
  journal = {Mathematical Programming Computation},
  year    = {2026},
  doi     = {10.1007/s12532-026-00320-7}
}

@inproceedings{HeEtAl2016ResNet,
  author    = {He, Kaiming and Zhang, Xiangyu and Ren, Shaoqing and Sun, Jian},
  title     = {Deep Residual Learning for Image Recognition},
  booktitle = {Proceedings of the IEEE Conference on Computer Vision and Pattern Recognition},
  pages     = {770--778},
  year      = {2016},
  doi       = {10.1109/CVPR.2016.90}
}

@techreport{Krizhevsky2009CIFAR,
  author      = {Krizhevsky, Alex},
  title       = {Learning Multiple Layers of Features from Tiny Images},
  institution = {University of Toronto},
  year        = {2009}
}

@article{LeeEtAl2023ApproxCNN,
  author  = {Lee, Junghyun and Lee, Eunsang and Lee, Joon-Woo and Kim, Yongjune and Kim, Young-Sik and No, Jong-Seon},
  title   = {Precise Approximation of Convolutional Neural Networks for Homomorphically Encrypted Data},
  journal = {IEEE Access},
  volume  = {11},
  pages   = {62062--62076},
  year    = {2023},
  doi     = {10.1109/ACCESS.2023.3287564}
}

@article{Nesterov2005Smoothing,
  author  = {Nesterov, Yurii},
  title   = {Smooth Minimization of Non-Smooth Functions},
  journal = {Mathematical Programming},
  volume  = {103},
  number  = {1},
  pages   = {127--152},
  year    = {2005},
  doi     = {10.1007/s10107-004-0552-5}
}

@article{ByrdEtAl1995LBFGSB,
  author  = {Byrd, Richard H. and Lu, Peihuang and Nocedal, Jorge and Zhu, Ciyou},
  title   = {A Limited Memory Algorithm for Bound Constrained Optimization},
  journal = {SIAM Journal on Scientific Computing},
  volume  = {16},
  number  = {5},
  pages   = {1190--1208},
  year    = {1995},
  doi     = {10.1137/0916069}
}

@inproceedings{TongEtAl2024SmartPAF,
  author    = {Tong, Jianming and Dang, Jingtian and Golder, Anupam and Raychowdhury, Arijit and Hao, Cong and Krishna, Tushar},
  title     = {Accurate Low-Degree Polynomial Approximation of Non-Polynomial Operators for Fast Private Inference in Homomorphic Encryption},
  booktitle = {Proceedings of Machine Learning and Systems},
  volume    = {6},
  pages     = {210--223},
  year      = {2024}
}

@article{ZhaoTay2023VFD,
  author  = {Zhao, Ruijie and Tay, David B.},
  title   = {A Complex Exponential Structure for Low-Complexity Variable Fractional Delay {FIR} Filters},
  journal = {Circuits, Systems, and Signal Processing},
  volume  = {42},
  number  = {2},
  pages   = {1105--1141},
  year    = {2023},
  doi     = {10.1007/s00034-022-02169-2}
}

@article{Rabiner1972FIRLP,
  author  = {Rabiner, Lawrence R.},
  title   = {The Design of Finite Impulse Response Digital Filters Using Linear Programming Techniques},
  journal = {Bell System Technical Journal},
  volume  = {51},
  number  = {6},
  pages   = {1177--1198},
  year    = {1972},
  doi     = {10.1002/j.1538-7305.1972.tb02649.x}
}

@article{McClellanParksRabiner1973,
  author  = {McClellan, James H. and Parks, Thomas W. and Rabiner, Lawrence R.},
  title   = {A Computer Program for Designing Optimum {FIR} Linear Phase Digital Filters},
  journal = {IEEE Transactions on Audio and Electroacoustics},
  volume  = {21},
  number  = {6},
  pages   = {506--526},
  year    = {1973},
  doi     = {10.1109/TAU.1973.1162525}
}

@article{BertsimasEtAl2011Robust,
  author  = {Bertsimas, Dimitris and Brown, David B. and Caramanis, Constantine},
  title   = {Theory and Applications of Robust Optimization},
  journal = {SIAM Review},
  volume  = {53},
  number  = {3},
  pages   = {464--501},
  year    = {2011},
  doi     = {10.1137/080734510}
}

@inproceedings{ShalevShwartzWexler2016,
  author    = {Shalev-Shwartz, Shai and Wexler, Yonatan},
  title     = {Minimizing the Maximal Loss: How and Why},
  booktitle = {Proceedings of the 33rd International Conference on Machine Learning},
  series    = {Proceedings of Machine Learning Research},
  volume    = {48},
  pages     = {793--801},
  year      = {2016}
}

@inproceedings{SagawaEtAl2020GroupDRO,
  author    = {Sagawa, Shiori and Koh, Pang Wei and Hashimoto, Tatsunori B. and Liang, Percy},
  title     = {Distributionally Robust Neural Networks for Group Shifts: On the Importance of Regularization for Worst-Case Generalization},
  booktitle = {International Conference on Learning Representations},
  year      = {2020}
}

@inproceedings{YuEtAl2024GroupDRO,
  author    = {Yu, Dingzhi and Cai, Yunuo and Jiang, Wei and Zhang, Lijun},
  title     = {Efficient Algorithms for Empirical Group Distributionally Robust Optimization and Beyond},
  booktitle = {Proceedings of the 41st International Conference on Machine Learning},
  series    = {Proceedings of Machine Learning Research},
  volume    = {235},
  pages     = {57384--57414},
  year      = {2024}
}

@inproceedings{GiladBachrachEtAl2016CryptoNets,
  author    = {Gilad-Bachrach, Ran and Dowlin, Nathan and Laine, Kim and Lauter, Kristin and Naehrig, Michael and Wernsing, John},
  title     = {{CryptoNets}: Applying Neural Networks to Encrypted Data with High Throughput and Accuracy},
  booktitle = {Proceedings of the 33rd International Conference on Machine Learning},
  series    = {Proceedings of Machine Learning Research},
  volume    = {48},
  pages     = {201--210},
  year      = {2016}
}

@inproceedings{CarmonEtAl2021,
  author    = {Carmon, Yair and Jambulapati, Arun and Jin, Yujia and Sidford, Aaron},
  title     = {Thinking Inside the Ball: Near-Optimal Minimization of the Maximal Loss},
  booktitle = {Proceedings of the 34th Conference on Learning Theory},
  series    = {Proceedings of Machine Learning Research},
  volume    = {134},
  pages     = {866--882},
  year      = {2021}
}

@article{PolakRoysetWomersley2003,
  author  = {Polak, Elijah and Royset, Johannes O. and Womersley, Robert S.},
  title   = {Algorithms with Adaptive Smoothing for Finite Minimax Problems},
  journal = {Journal of Optimization Theory and Applications},
  volume  = {119},
  number  = {3},
  pages   = {459--484},
  year    = {2003},
  doi     = {10.1023/B:JOTA.0000006685.60019.3e}
}

@article{PeeRoyset2011,
  author  = {Pee, Eng Yau and Royset, Johannes O.},
  title   = {On Solving Large-Scale Finite Minimax Problems Using Exponential Smoothing},
  journal = {Journal of Optimization Theory and Applications},
  volume  = {148},
  number  = {2},
  pages   = {390--421},
  year    = {2011},
  doi     = {10.1007/s10957-010-9759-1}
}

@article{BeckTeboulle2012,
  author  = {Beck, Amir and Teboulle, Marc},
  title   = {Smoothing and First Order Methods: A Unified Framework},
  journal = {SIAM Journal on Optimization},
  volume  = {22},
  number  = {2},
  pages   = {557--580},
  year    = {2012},
  doi     = {10.1137/100818327}
}

@article{RasTamUteda2025,
  author  = {Ras, Charl J. and Tam, Matthew K. and Uteda, Daniel J.},
  title   = {Identification of Active Component Functions in Finite-Max Minimisation via a Smooth Reformulation},
  journal = {Applied Mathematics \& Optimization},
  volume  = {91},
  number  = {2},
  pages   = {36},
  year    = {2025},
  doi     = {10.1007/s00245-025-10229-7}
}

@article{Nemirovski2004MirrorProx,
  author  = {Nemirovski, Arkadi},
  title   = {Prox-Method with Rate of Convergence {$O(1/t)$} for Variational Inequalities with {Lipschitz} Continuous Monotone Operators and Smooth Convex--Concave Saddle Point Problems},
  journal = {SIAM Journal on Optimization},
  volume  = {15},
  number  = {1},
  pages   = {229--251},
  year    = {2004},
  doi     = {10.1137/S1052623403425629}
}

@article{NemirovskiEtAl2009SA,
  author  = {Nemirovski, Arkadi and Juditsky, Anatoli and Lan, Guanghui and Shapiro, Alexander},
  title   = {Robust Stochastic Approximation Approach to Stochastic Programming},
  journal = {SIAM Journal on Optimization},
  volume  = {19},
  number  = {4},
  pages   = {1574--1609},
  year    = {2009},
  doi     = {10.1137/070704277}
}

@article{HamedaniAybat2021,
  author  = {Hamedani, Erfan Yazdandoost and Aybat, Necdet Serhat},
  title   = {A Primal-Dual Algorithm with Line Search for General Convex--Concave Saddle Point Problems},
  journal = {SIAM Journal on Optimization},
  volume  = {31},
  number  = {2},
  pages   = {1299--1329},
  year    = {2021},
  doi     = {10.1137/18M1213488}
}

@article{ZhuLiuTranDinh2022,
  author  = {Zhu, Yuzixuan and Liu, Deyi and Tran-Dinh, Quoc},
  title   = {New Primal-Dual Algorithms for a Class of Nonsmooth and Nonlinear Convex--Concave Minimax Problems},
  journal = {SIAM Journal on Optimization},
  volume  = {32},
  number  = {4},
  pages   = {2580--2611},
  year    = {2022},
  doi     = {10.1137/21M1408683}
}

@inproceedings{AlacaogluMalitsky2022,
  author    = {Alacaoglu, Ahmet and Malitsky, Yura},
  title     = {Stochastic Variance Reduction for Variational Inequality Methods},
  booktitle = {Proceedings of Thirty Fifth Conference on Learning Theory},
  series    = {Proceedings of Machine Learning Research},
  volume    = {178},
  pages     = {778--816},
  year      = {2022}
}

@misc{GladinEtAl2025LSE,
  author  = {Gladin, Egor and Kroshnin, Alexey and Zhu, Jia-Jie and Dvurechensky, Pavel},
  title   = {Improved Stochastic Optimization of {LogSumExp}},
  year    = {2025},
  doi     = {10.48550/arXiv.2509.24894}
}

@article{SamakhoanaGrimmer2025,
  author  = {Samakhoana, Thabo and Grimmer, Benjamin},
  title   = {An Elementary Proof of the Near Optimality of {LogSumExp} Smoothing},
  journal = {arXiv preprint arXiv:2512.10825},
  year    = {2025},
  doi     = {10.48550/arXiv.2512.10825}
}

@article{GaudiosoEtAl2006,
  author  = {Gaudioso, Manlio and Giallombardo, Giovanni and Miglionico, Giovanna},
  title   = {An Incremental Method for Solving Convex Finite Min-Max Problems},
  journal = {Mathematics of Operations Research},
  volume  = {31},
  number  = {1},
  pages   = {173--187},
  year    = {2006},
  doi     = {10.1287/moor.1050.0175}
}

@article{ZhouTits1996,
  author  = {Zhou, Jian L. and Tits, Andr{\'e} L.},
  title   = {An {SQP} Algorithm for Finely Discretized Continuous Minimax Problems and Other Minimax Problems with Many Objective Functions},
  journal = {SIAM Journal on Optimization},
  volume  = {6},
  number  = {2},
  pages   = {461--487},
  year    = {1996},
  doi     = {10.1137/0806025}
}

@article{ZhangEtAl2010Exchange,
  author  = {Zhang, Liping and Wu, Soon-Yi and L{\'o}pez, Marco A.},
  title   = {A New Exchange Method for Convex Semi-Infinite Programming},
  journal = {SIAM Journal on Optimization},
  volume  = {20},
  number  = {6},
  pages   = {2959--2977},
  year    = {2010},
  doi     = {10.1137/090767133}
}

@inproceedings{KawaguchiLu2020,
  author    = {Kawaguchi, Kenji and Lu, Haihao},
  title     = {Ordered {SGD}: A New Stochastic Optimization Framework for Empirical Risk Minimization},
  booktitle = {Proceedings of the 23rd International Conference on Artificial Intelligence and Statistics},
  series    = {Proceedings of Machine Learning Research},
  volume    = {108},
  pages     = {669--679},
  year      = {2020}
}

@inproceedings{FanEtAl2017TopK,
  author    = {Fan, Yanbo and Lyu, Siwei and Ying, Yiming and Hu, Bao-Gang},
  title     = {Learning with Average Top-$k$ Loss},
  booktitle = {Advances in Neural Information Processing Systems},
  volume    = {30},
  pages     = {497--505},
  year      = {2017},
  publisher = {Curran Associates, Inc.}
}
\bibliographystyle{plainnat}

\clearpage
\appendix
\section{Oracle Model and Surrogate Subgradients}
\label{app:oracle}

For every component, let \(\sigma_i:\X\to[0,\infty)\) be a deterministic
measurable function.  Whenever \(x\in\X\) and \(i\in[N]\) are measurable with
respect to a sigma-field \(\mathcal H\) before a fresh oracle draw, assume
\begin{equation}
    \mathbb E[g_i(x;\zeta)\mid\mathcal H]\in\partial f_i(x),
    \qquad
    \mathbb E[\|g_i(x;\zeta)\|_2^2\mid\mathcal H]\le\sigma_i^2(x)
    \quad\text{a.s.}
    \label{eq:app-oracle}
\end{equation}
Thus \(\sigma_i^2(x)\) is a componentwise conditional second-moment bound.
The component-value oracle returns \(f_i(x)\) exactly.  Define
\begin{equation}
    G_{\rm rms}^2
    \defeq
    \sup_{x\in\X}\frac1N\sum_{i=1}^N\sigma_i^2(x),
    \qquad
    G^2
    \defeq
    \sup_{x\in\X}\max_i\sigma_i^2(x).
    \label{eq:app-moment-scales}
\end{equation}
Thus \(G\) is the uniform moment bound used in the main text.  The RMS
condition \(G_{\rm rms}<\infty\) also implies \(G^2\le NG_{\rm rms}^2\).

\subsection{Surrogate subgradient and selection-weighted moment bound}
\label{app:surrogate-gradient}

Let \(I(S,x)\) be the deterministically tie-broken maximizer of the sampled
components and put
\begin{equation}
    \pi_{m,i}(x)\defeq\mathbb P_S(I(S,x)=i),
    \qquad
    G_m^2\defeq
    \sup_{x\in\X}\sum_{i=1}^N\pi_{m,i}(x)\sigma_i^2(x).
    \label{eq:app-selected-moment}
\end{equation}

\begin{lemma}
\label{lem:app-surrogate-gradient}
For Algorithm~\ref{alg:smax},
\[
    \mathbb E[g_{I_t}(x_t;\zeta_t)\mid\mathcal H_t]
    \in\partial\widetilde F_m(x_t),
    \qquad
    \mathbb E[\|g_{I_t}(x_t;\zeta_t)\|_2^2\mid\mathcal H_t]
    \le G_m^2,
\]
where \(\mathcal H_t\) is the history before sampling \(S_t\).  Moreover,
\begin{equation}
    G_m^2\le\min\{G^2,\,mG_{\rm rms}^2\}.
    \label{eq:app-moment-comparison}
\end{equation}
\end{lemma}

\begin{proof}
Note that \(I_t=I(S_t,x_t)\), and define for each size-\(m\) subset \(S\)
\[
    v_S(x_t)
    \defeq
    \mathbb E\!\left[
        g_{I(S,x_t)}(x_t;\zeta_t)
        \mid \mathcal H_t,S_t=S
    \right].
\]
The selected component \(I(S,x_t)\) is active in \(F_S(x_t)\).  Therefore
\eqref{eq:app-oracle} and the subdifferential rule for a finite maximum give
\[
    v_S(x_t)\in\partial f_{I(S,x_t)}(x_t)
    \subseteq\partial F_S(x_t).
\]
The tower property and the \eqref{eq:smax-surrogate} yield
\[
	\mathbb E[g_{I_t}(x_t;\zeta_t)\mid\mathcal H_t]
	=\mathbb E\!\left[
	\mathbb E[g_{I(S_t,x_t)}(x_t;\zeta_t)
	\mid\mathcal H_t,S_t]
	\mathrel{\big|}\mathcal H_t
	\right]
	=\mathbb E\!\left[
	v_S(x_t)
	\mid\mathcal H_t
	\right] \in \partial\widetilde F_m(x_t).
\]

For the second moment, another application of the tower property gives
\begin{align*}
    \mathbb E[\|g_{I_t}(x_t;\zeta_t)\|_2^2\mid\mathcal H_t]
    &=\mathbb E\!\left[
        \mathbb E[\|g_{I(S_t,x_t)}(x_t;\zeta_t)\|_2^2
                  \mid\mathcal H_t,S_t]
        \mathrel{\big|}\mathcal H_t
      \right] \\
    &\stackrel{\eqref{eq:app-oracle}}{\le}
      \mathbb E[\sigma_{I(S_t,x_t)}^2(x_t)\mid\mathcal H_t] \\
    &=\sum_{i=1}^N
      \mathbb P(I(S_t,x_t)=i\mid\mathcal H_t)\sigma_i^2(x_t) \\
    &\stackrel{\eqref{eq:app-selected-moment}}{=}\sum_{i=1}^N\pi_{m,i}(x_t)\sigma_i^2(x_t)
      \le G_m^2.
\end{align*}
Finally, \(\pi_{m,i}(x)\le\mathbb P(i\in S)=m/N\) and
\(\sum_i\pi_{m,i}(x)=1\).  Applying these two bounds separately inside
\eqref{eq:app-selected-moment} proves \eqref{eq:app-moment-comparison}.
\end{proof}

In particular, \(G_m\le G\), so this lemma proves
Lemma~\ref{lem:surrogate-gradient}.

\section{Sampled-Surrogate Approximation}
\label{app:approximation-proof}

\begin{proof}[Proof of Proposition~\ref{prop:approximation}]
Fix \(x\) and write
\(\Delta_j=f_{(1)}(x)-f_{(j)}(x)\), so
\(0=\Delta_1\le\cdots\le\Delta_N\).  Let \(J\) be the best rank sampled by
\(S\); then
\[
    F(x)-\widetilde F_m(x)=\mathbb E[\Delta_J].
\]
For \(j\le k\), the event \(J=j\) requires rank \(j\), excludes the better
ranks, and chooses the remaining \(m-1\) elements below rank \(j\).  Thus
\[
    \mathbb P(J=j)=\frac{\binom{N-j}{m-1}}{\binom Nm},
    \qquad
    \mathbb P(J>k)=\pi^{\rm miss}_{m,k}.
\]
Conditional on \(J\le k\), the probability assigned to rank \(j\) is proportional to
\(\binom{N-j}{m-1}\). These weights are nonincreasing in \(j\), so the conditional distribution puts at least as much weight on the smaller ranks as the uniform distribution on \(\{1,\ldots,k\}\). Because \(\Delta_j\) is nondecreasing, Chebyshev’s inequality for oppositely ordered sequences (see, e.g., \citep[Section 2.17]{hardy1952inequalities}) yields
\[
    \mathbb E[\Delta_J\mid J\le k]
    \le\frac1k\sum_{j=1}^k\Delta_j
    =F(x)-\operatorname{AT}_k(x).
\]
On \(J>k\), the loss is at most \(f_{(1)}(x)-f_{(N)}(x)\), proving
\eqref{eq:pointwise-approximation}.  Finally,
\[
    \pi^{\rm miss}_{m,k}
    =\prod_{r=0}^{m-1}\frac{N-k-r}{N-r}
    \le\left(1-\frac kN\right)^m
    \le e^{-mk/N},
\]
which gives \eqref{eq:regional-approximation}.
\end{proof}

\section{Localization to a True-Objective Sublevel Set}
\label{app:localization}

\begin{lemma}
\label{lem:app-barrier}
If \(\operatorname{err}_{m,k}(\rho)<\rho\), then
\begin{equation}
    x\notin\X_\rho
    \quad\Longrightarrow\quad
    \operatorname{subopt}_m(x)>\rho-\operatorname{err}_{m,k}(\rho).
    \label{eq:app-outside-barrier}
\end{equation}
Consequently, if
\(\operatorname{subopt}_m(x)\le\rho-\operatorname{err}_{m,k}(\rho)\), then
\begin{equation}
    x\in\X_\rho,
    \qquad
    F(x)-F^\star
    \le\operatorname{subopt}_m(x)+\operatorname{err}_{m,k}(\rho)
    \le\rho.
    \label{eq:app-local-transfer}
\end{equation}
\end{lemma}

\begin{proof}
Let \(x^\star\in\arg\min_\X F\).  If \(x\notin\X_\rho\), continuity along
the segment from \(x^\star\) to \(x\) gives
\(z=(1-\lambda)x^\star+\lambda x\), with \(\lambda\in(0,1)\) and
\(F(z)=F^\star+\rho\).  Proposition~\ref{prop:approximation} on
\(\X_\rho\) yields 
\[
    \widetilde F_m(z)\ge F^\star+\rho-\operatorname{err}_{m,k}(\rho).
\]
On the other hand, \(\widetilde F_m(x^\star)\le F^\star\),
\(\widetilde F_m^\star\le F^\star\), and convexity give
\begin{align*}
    \widetilde F_m(z)
    &\le(1-\lambda)\widetilde F_m(x^\star)+\lambda\widetilde F_m(x)\\
    &=(1-\lambda)\widetilde F_m(x^\star)
      +\lambda\bigl(\widetilde F_m^\star+\operatorname{subopt}_m(x)\bigr)\\
    &\le F^\star+\lambda\operatorname{subopt}_m(x).
\end{align*}
Therefore
\(\rho-\operatorname{err}_{m,k}(\rho)
\le\lambda\operatorname{subopt}_m(x)<\operatorname{subopt}_m(x)\), proving
\eqref{eq:app-outside-barrier}.  Its contrapositive gives the containment in
\eqref{eq:app-local-transfer}; inside \(\X_\rho\),
\[
    F(x)-F^\star
    \le F(x)-\widetilde F_m^\star
    = F(x)-\widetilde F_m(x)+\operatorname{subopt}_m(x)
    \le \operatorname{err}_{m,k}(\rho)+\operatorname{subopt}_m(x).
\]
\end{proof}

\begin{proof}[Proof of Theorem~\ref{thm:localized-transfer}]
Lemma~\ref{lem:app-barrier} implies
\[
\{\widehat x\notin\X_\rho\}
\subseteq
\left\{
\operatorname{subopt}_m(\widehat x)>
\rho-\operatorname{err}_{m,k}(\rho)
\right\}.
\]
Therefore, by Markov's inequality and
\(\mathbb E[\operatorname{subopt}_m(\widehat x)]\le\beta\),
\[
\mathbb P(\widehat x\notin\X_\rho)
\le
\frac{\beta}
{\rho-\operatorname{err}_{m,k}(\rho)},
\]
which proves \eqref{eq:localized-transfer-probability}.

For the expectation bound, split according to whether
\(\widehat x\in\X_\rho\).  On \(\X_\rho\), the localized surrogate-approximation bound, together with
\(\widetilde F_m^\star\le F^\star\), gives
$$
\begin{aligned}
	F(x)-F^{\star} & \leq F(x)-\widetilde{F}_m^{\star} \\
	& =F(x)-\widetilde{F}_m(x)+\operatorname{subopt}_m(x) \\
	& \leq \operatorname{err}_{m, k}(\rho)+\operatorname{subopt}_m(x) .
\end{aligned}
$$
On the complementary event, Lipschitz continuity of \(F\) and the diameter
bound yield
\[
F(\widehat x)-F^\star
\le
G\|\widehat x-x^\star\|_2
\le
GD_{\X},
\]
where \(x^\star\in\arg\min_\X F\).  Hence
\begin{align*}
	\mathbb E[F(\widehat x)-F^\star]
	&\le
	\mathbb E\!\left[
	\bigl(\operatorname{subopt}_m(\widehat x)
	+\operatorname{err}_{m,k}(\rho)\bigr)
	\mathbf 1_{\{\widehat x\in\X_\rho\}}
	\right]
	+GD_{\X}\,\mathbb P(\widehat x\notin\X_\rho)\\
	&\le
	\beta+\operatorname{err}_{m,k}(\rho)
	+GD_{\X}\,\mathbb P(\widehat x\notin\X_\rho).
\end{align*}
Combining the preceding Markov bound with the trivial bound
\(\mathbb P(\widehat x\notin\X_\rho)\le1\) gives
\[
\mathbb E[F(\widehat x)-F^\star]
\le
\operatorname{err}_{m,k}(\rho)+\beta
+GD_{\X}
\min\!\left\{
1,\,
\frac{\beta}{\rho-\operatorname{err}_{m,k}(\rho)}
\right\},
\]
which proves \eqref{eq:localized-transfer-expectation}.
\end{proof}

\section{Optimization and Query Complexity}
\label{app:complexity}

\begin{lemma}
\label{lem:app-sgd}
For any \(T\ge1\) and constant stepsize \(\eta_t\equiv\eta>0\), the output of
Algorithm~\ref{alg:smax} satisfies
\begin{equation}
    \mathbb E[\operatorname{subopt}_m(\bar x_T)]
    \le\frac{D_{\X}^2}{2\eta T}+\frac{\eta G_m^2}{2}.
    \label{eq:app-sgd-bound}
\end{equation}
\end{lemma}

\begin{proof}
Let \(x_m^\star\in\arg\min_\X\widetilde F_m\).  Nonexpansiveness of projection,
Lemma~\ref{lem:app-surrogate-gradient}, and the subgradient inequality imply
\begin{align*}
    \mathbb E_t\|x_{t+1}-x_m^\star\|_2^2
    &\le\|x_t-x_m^\star\|_2^2
      -2\eta\bigl[\widetilde F_m(x_t)-\widetilde F_m^\star\bigr]
      +\eta^2G_m^2.
\end{align*}
Summing, using \(\|x_1-x_m^\star\|_2\le D_{\X}\), and applying convexity to the
uniform average proves \eqref{eq:app-sgd-bound}.
\end{proof}

We first prove a selection-weighted version of the localized complexity result stated in the main text, and then turn to Theorem~\ref{thm:localized-complexity} itself.

\begin{theorem}
\label{thm:app-complexity}
For a target \(\varepsilon>0\) and confidence parameter \(\delta\in(0,1)\),
suppose that some \(k\in[N]\) satisfies
\(\operatorname{gap}_k(\X_\varepsilon)\le\varepsilon/4\).  Choose \(m\) as in
\eqref{eq:subset-choice}.  Assume \(D_{\X}G_m>0\) and set
\[
    T=
      \left\lceil\frac{4D_{\X}^2G_m^2}{\delta^2\varepsilon^2}\right\rceil,
\]
and run Algorithm~\ref{alg:smax} with
\(\eta_t\equiv D_{\X}/(G_m\sqrt T)\).  Its output satisfies
\begin{equation}
    \mathbb P\!\left(F(\bar x_T)-F^\star\le\varepsilon\right)\ge1-\delta.
    \label{eq:app-final-guarantee}
\end{equation}
Consequently, at fixed confidence,
the numbers of stochastic-subgradient and component-value oracle calls used by
the method satisfy
\begin{align}
    \mathsf Q_{\rm grad}
    &=O\!\left(1+\frac{D_{\X}^2G_m^2}{\varepsilon^2}\right),
    \label{eq:app-grad-complexity}\\
    \mathsf Q_{\rm val}
    &=O\!\left(
      \left[1+\frac{D_{\X}^2G_m^2}{\varepsilon^2}\right]
      \min\!\left\{N-k+1,
      1+\frac Nk\log\!\left(1+\frac{W(\X_\varepsilon)}{\varepsilon}\right)
      \right\}\right).
    \label{eq:app-val-complexity}
\end{align}
\end{theorem}

If \(D_{\X}G_m=0\), the conclusion is immediate: either \(\X\) is a
singleton or the sampled surrogate is constant, and the transfer theorem
applies with zero surrogate suboptimality.

\begin{proof}
Put \(W_\varepsilon\defeq W(\X_\varepsilon)\).
Unless the outer minimum in \eqref{eq:subset-choice} selects \(N-k+1\), the bound \(\pi^{\rm miss}_{m,k}\le\exp(-mk/N)\) from Proposition~\ref{prop:approximation} gives
\[
    W_\varepsilon\pi^{\rm miss}_{m,k}
    \le\frac{W_\varepsilon}{1+4W_\varepsilon/\varepsilon}
    \le\frac\varepsilon4.
\]
If \(m=N-k+1\), the binomial numerator is zero.  Thus
\(\operatorname{err}_{m,k}(\varepsilon)\le\varepsilon/2\).
Lemma~\ref{lem:app-sgd}
with \(\eta=D_{\X}/(G_m\sqrt T)\) yields
\[
\mathbb E[\operatorname{subopt}_m(\bar x_T)]
\le \frac{D_{\X}G_m}{\sqrt T}
\le \frac{\delta\varepsilon}{2}.
\]
Theorem~\ref{thm:localized-transfer} now proves \eqref{eq:app-final-guarantee}.  Each
iteration uses one subgradient and \(m\) values, which gives
\eqref{eq:app-grad-complexity}--\eqref{eq:app-val-complexity}.
\end{proof}

\begin{proof}[Proof of Theorem~\ref{thm:localized-complexity}]
The subset choice gives
\(\operatorname{err}_{m,k}(\varepsilon)\le\varepsilon/2\) by the same
argument as above.  Under the uniform main-text moment bound, Lemma~\ref{lem:app-sgd}
with \(\eta=D_{\X}/(G\sqrt T)\) and \(G_m\le G\) yields
\[
    \mathbb E[\operatorname{subopt}_m(\bar x_T)]
    \le \frac{D_{\X}G}{\sqrt T}
    \le \frac{\delta\varepsilon}{2}.
\]
Theorem~\ref{thm:localized-transfer} now gives the stated probability.  Finally,
each iteration uses one subgradient query and \(m\) component-value queries;
substituting \eqref{eq:subset-choice} gives \eqref{eq:query-complexity}.
\end{proof}

\begin{remark}
	As a consequence of Theorem~\ref{thm:app-complexity} and
	Lemma~\ref{lem:app-surrogate-gradient}, under only the RMS condition we have
	\(G_m^2\le mG_{\rm rms}^2\), and hence, at fixed confidence,
	\begin{equation}
		\mathsf Q_{\rm grad}
		=O\!\left(
		1+\frac{D_{\X}^2mG_{\rm rms}^2}{\varepsilon^2}
		\right),
		\qquad
		\mathsf Q_{\rm val}
		=O\!\left(
		m+\frac{D_{\X}^2m^2G_{\rm rms}^2}{\varepsilon^2}
		\right).
		\label{eq:app-averaged-complexity}
	\end{equation}
	Thus the RMS condition alone can introduce an additional factor of \(m\) in
	the value-query complexity.  In particular, the choice
	\(m=\widetilde O(N/k)\) does not by itself imply
	\(\widetilde O(N/k)\) value-query dependence; such a bound additionally
	requires the selection-weighted moment \(G_m\) to remain suitably controlled.
\end{remark}

\section{Value-Query Lower Bound}
\label{app:value-lower-bound}

\begin{proof}[Proof of Theorem~\ref{thm:value-lower-bound}]
Partition the indices into two known sets \(P_+\) and \(P_-\), each of size
\(N/2\), and set
\[
    \sigma_i=+1\quad(i\in P_+),
    \qquad
    \sigma_i=-1\quad(i\in P_-).
\]
Choose a hidden sign \(B\) uniformly from \(\{+1,-1\}\), and conditional on
\(B\), choose a uniform \(k\)-element subset \(S_B\subseteq P_B\).  For
\(x\in[-R,R]\), define
\begin{equation}
    f_i(x)=\sigma_iGx+4GR\,\mathbf 1\{i\in S_B\}.
    \label{eq:app-lower-bound-components}
\end{equation}
Every component is affine with its unique subgradient equal to
\(\sigma_iG\).  Consequently, subgradient responses are independent of both
\(B\) and \(S_B\), no matter how many such queries are made.

For \(i\in S_B\), equation~\eqref{eq:app-lower-bound-components} gives
\(f_i(x)=BGx+4GR\ge3GR\), whereas every nonspecial component has value at most
\(GR\).  Hence the \(k\) special components are precisely the maximizers and
\begin{equation}
    F_B(x)=4GR+BGx,
    \qquad
    \operatorname{gap}_k(\X)=0.
    \label{eq:app-lower-bound-objective}
\end{equation}
All component values lie between \(-GR\) and \(5GR\), which also proves
\(W(\X)\le6GR\).  Moreover,
\[
    F_B^\star=3GR,
    \qquad
    F_B(x)-F_B^\star=G(R+Bx).
\]
If \(\varepsilon\le GR/2\), an \(\varepsilon\)-accurate point must therefore
satisfy \(x\le-R/2\) when \(B=+1\), and \(x\ge R/2\) when \(B=-1\).  Thus the
sign of an accurate output identifies \(B\).

It remains to bound the information supplied by value queries.  Couple the
experiments \(B=+1\) and \(B=-1\) by drawing independent uniform \(k\)-subsets
\(S_+\subset P_+\) and \(S_-\subset P_-\).  First fix the algorithm's internal
randomness, and consider the common oracle history---the sequence of queries
and oracle responses---generated when every value query \((x,i)\) receives the
nonspecial response \(\sigma_iGx\).  Let
\(A_+\subseteq P_+\) and \(A_-\subseteq P_-\) be the distinct indices queried
along this history.  If
\[
    S_+\cap A_+=\varnothing
    \quad\text{and}\quad
    S_-\cap A_-=\varnothing,
\]
then the algorithm receives the same value and subgradient responses, makes
the same subsequent queries, and returns the same output in the two
experiments.  Since a fixed index in either half belongs to its uniform special
set with probability \(2k/N\), a union bound gives
\[
    \mathbb P(\text{the coupled oracle histories differ})
    \le \frac{2k}{N}\bigl(|A_+|+|A_-|\bigr)
    \le \frac{2k\mathsf Q_{\rm val}}{N}.
\]
The same coupling is valid after averaging over the algorithm's randomness.
Therefore the total-variation distance between the two distributions of the
complete oracle history and output is at most
\(2k\mathsf Q_{\rm val}/N\).  Under the uniform prior on \(B\), no test based
on this information can identify \(B\) with probability greater than
\[
    \frac12+\frac{k\mathsf Q_{\rm val}}{N}.
\]
An algorithm that is \(\varepsilon\)-accurate with probability at least \(2/3\)
on every instance induces, by the sign rule above, a test with average success
probability at least \(2/3\).  Consequently
\(\frac12+k\mathsf Q_{\rm val}/N\ge\frac23\), which is exactly
\eqref{eq:value-lower-bound}.
\end{proof}

\section{Tensor-Grid Derivation}
\label{app:grid-details}

\begin{proposition}
\label{prop:fine-grid-full}
Under~\eqref{eq:chebyshev-components}, for integers
\(a_\ell\in\{0,\ldots,n_\ell-1\}\), define
\[
    k(a)\defeq\prod_{\ell=1}^q(a_\ell+1),
    \qquad
    d(a)\defeq\max_{1\le\ell\le q}\frac{a_\ell}{n_\ell-1}.
\]
Then
\begin{equation}
    \operatorname{gap}_{k(a)}(\X)\le L_\tau d(a),
    \qquad W(\X)\le L_\tau.
    \label{eq:grid-gap}
\end{equation}
Writing \(H(x)=\sup_{\tau\in[0,1]^q}r_x(\tau)\), for every \(x\in\X\),
\begin{equation}
    0\le H(x)-F(x)\le
    \Delta_{\mathcal T}\defeq\frac{L_\tau}{2}
      \max_\ell\frac1{n_\ell-1}.
    \label{eq:grid-discretization}
\end{equation}
Thus, if \(F(\widehat x)\le\min_{x\in\X}F(x)+\xi\), then
\begin{equation}
    H(\widehat x)-\min_{x\in\X}H(x)\le\xi+\Delta_{\mathcal T}.
    \label{eq:grid-continuous-transfer}
\end{equation}
For \(\varepsilon>0\), define
\(A_\varepsilon\defeq 1+4L_\tau/\varepsilon\).  If \(L_\tau>0\), take
\begin{equation}
    a_\ell
    =\min\!\left\{n_\ell-1,\;
      \left\lfloor\frac{\varepsilon(n_\ell-1)}{4L_\tau}\right\rfloor
      \right\},
    \qquad k_\varepsilon=k(a),
    \label{eq:grid-K}
\end{equation}
and if \(L_\tau=0\), take \(a_\ell=n_\ell-1\).  Then
\begin{equation}
    \operatorname{gap}_{k_\varepsilon}(\X)\le\frac{\varepsilon}{4},
    \qquad
    \frac{N}{k_\varepsilon}\le A_\varepsilon^q.
    \label{eq:grid-epsilon-bound}
\end{equation}
\end{proposition}

\begin{proof}
Fix \(x\) and let \(\tau_{\boldsymbol i^\star}\) maximize \(r_x\) on the tensor
grid.  In coordinate \(\ell\), at least \(a_\ell+1\) grid indices lie within
\(a_\ell\) index steps of \(i_\ell^\star\).  Their Cartesian product therefore
contains at least \(k(a)\) grid points within \(\ell_\infty\)-distance
\(d(a)\) of \(\tau_{\boldsymbol i^\star}\).  At each such point,
\[
    r_x(\tau_{\boldsymbol i})\ge F(x)-L_\tau d(a).
\]
Averaging the largest \(k(a)\) residuals proves the first inequality in
\eqref{eq:grid-gap}.  The \(\ell_\infty\)-diameter of \([0,1]^q\) is one, so
the same Lipschitz bound gives \(W(\X)\le L_\tau\).

Every \(\tau\in[0,1]^q\) is within
\(\frac12\max_\ell(n_\ell-1)^{-1}\) of a tensor-grid point.  Comparing a
continuous maximizer with its nearest grid point proves
\eqref{eq:grid-discretization}.  If \(\widehat x\) is \(\xi\)-suboptimal for
\(F\), then
\[
 H(\widehat x)\le F(\widehat x)+\Delta_{\mathcal T}
 \le \min_\X F+\xi+\Delta_{\mathcal T}
 \le \min_\X H+\xi+\Delta_{\mathcal T},
\]
which proves the continuous-domain statement.

For \(L_\tau>0\), \eqref{eq:grid-K} gives
\(a_\ell/(n_\ell-1)\le\varepsilon/(4L_\tau)\), hence the first inequality
in~\eqref{eq:grid-epsilon-bound}.  Put
\(s=\varepsilon/(4L_\tau)\) and \(p_\ell=n_\ell-1\).  If \(s\ge1\), then
\(a_\ell=p_\ell\).  If \(s<1\), then
\(sp_\ell<a_\ell+1\), and therefore
\[
    \frac{n_\ell}{a_\ell+1}
    =\frac{p_\ell+1}{a_\ell+1}
    <\frac1s+\frac1{a_\ell+1}
    \le 1+\frac1s=A_\varepsilon.
\]
Multiplying over coordinates proves the second inequality
in~\eqref{eq:grid-epsilon-bound}.  If \(L_\tau=0\), taking every
\(a_\ell=n_\ell-1\) gives \(k_\varepsilon=N\), zero gap, and
\(A_\varepsilon=1\).
\end{proof}

\begin{proof}[Proof of Corollary~\ref{cor:fine-grid}]
The structural bounds follow directly from
Proposition~\ref{prop:fine-grid-full}.  Because
\(\X_\varepsilon\subseteq\X\), the same \(k_\varepsilon\) satisfies
\(\operatorname{gap}_{k_\varepsilon}(\X_\varepsilon)\le\varepsilon/4\) and
\(W(\X_\varepsilon)\le L_\tau\).  In~\eqref{eq:subset-choice}, use
\(N/k_\varepsilon\le A_\varepsilon^q\),
\(1+4W(\X_\varepsilon)/\varepsilon\le A_\varepsilon\), and
\(\lceil z\rceil\le z+1\).  This gives the displayed subset-size bound.
If component subgradients have norm at most \(G_r\), then the moment bound in
Theorem~\ref{thm:localized-complexity} satisfies \(G\le G_r\), and
\eqref{eq:query-complexity} gives~\eqref{eq:grid-complexity}.
\end{proof}

\section{Variable Fractional-Delay Experiment Details}
\label{app:vfd-experiment}

\subsection{Finite-max construction}

We use \(L_h=60\), polynomial order \(r=4\),
\(\omega\in[0,0.9\pi]\), and \(p\in[0,0.5]\) as in
\citet{ZhaoTay2023VFD}.  The response and target are given in
\eqref{eq:vfd-response}.  On the \(1000\times200\) training grid,
\begin{equation}
    f_{i\ell}(a)
    =\left\|
      \begin{pmatrix}\Re e_{i\ell}(a)\\ \Im e_{i\ell}(a)\end{pmatrix}
      \right\|_2,
    \qquad
    e_{i\ell}(a)=H_a(\omega_i,p_\ell)-D(\omega_i,p_\ell).
    \label{eq:app-vfd-components}
\end{equation}
Thus the training objective is
\(F_{\rm train}(a)=\max_{i,\ell}f_{i\ell}(a)\).
The symmetry-reduced coefficient vector has 153 real entries.  The
implementation uses the scaled coordinate \(u=2p\), a diagonal
reparameterization of the same degree-four polynomial space.  Every method
uses the common box \([-B,B]^{153}\), where
\(B=\max\{1,2\|a_{\rm WLS}\|_\infty\}=1.99988\), and starts from the same
tensor least-squares solution.  Component values and subgradients are
deterministic.

The denser validation grid contains \(2001\times401=802{,}401\) points.  It is
used only after configurations and output rules are fixed, solely to check the
reported responses.  Across the four reported outputs, its maximum
exceeds the training-grid maximum by at most 0.022\%.

\subsection{Optimizers and protocol}
\label{app:vfd-protocol}

Restricted-SOCP exchange solves the active complex-modulus constraints with
Clarabel 0.11.1 \citep{GoulartChen2026Clarabel}, separates on the entire training
grid, and produces the numerical reference interval
\([\ell,u]=[2.70495097,2.70495121]\times10^{-3}\).  We report
\begin{equation}
    \operatorname{relgap}_\ell(a)
    =\left[\frac{F_{\rm train}(a)-\ell}{\ell}\right]_+.
    \label{eq:app-vfd-gap}
\end{equation}
Because \(\ell\) comes from a numerical conic solve, we treat this as a gap to
a numerical reference rather than a rigorous certificate.  The solver's
residual is negligible relative to the 5\% headline threshold.

\SMax\ is selected using training values; every candidate run in the
development search is capped at the common horizon
\(\mathsf Q_{\rm dev}=25N=5\times10^6\).  The search on seeds 0--2 selects
\[
    m=16384,
    \qquad
    \eta_1=5.9909\times10^{-5},
    \qquad
    t_0=5.997,
    \qquad
    \eta_0=1.5848\times10^{-4},
\]
for \(\eta_t=\eta_0/\sqrt{t+t_0}\).  This configuration is frozen and run on
confirmatory seeds 200--219.  Full-grid SGM is tuned by the same endpoint
training criterion and horizon and uses
\(\eta_0=8.3198\times10^{-5}\), \(t_0=29.857\).  LSE--L-BFGS uses the fixed
continuation
\[
    \mu/F_{\rm train}(a_{\rm WLS})\in
    \{10^{-1},3\!\times\!10^{-2},10^{-2},3\!\times\!10^{-3},
      10^{-3},3\!\times\!10^{-4}\}
\].
This LogSumExp smoothing \citep{Nesterov2005Smoothing} is minimized by
L-BFGS-B \citep{ByrdEtAl1995LBFGSB}, with at most 35 iterations per stage.
Under the \(25N\) cap, LSE remains in the first temperature and returns the
smallest exact training maximum among its already charged trial points; this
is at least as favorable as returning only the last accepted iterate.  \SMax,
full-grid SGM, LSE, and the reported SOCP endpoint each receive at most
\(25N=5\times10^6\) values.  SOCP is continued to \(29N\) only to establish
the numerical reference used for offline evaluation.

One modulus evaluation at one \((\omega,p)\) pair counts as one component
value.  The LSE objective-gradient computation counts one full-grid pass, and
every SOCP separation scan counts \(N\) values.  The reported times and query
axes exclude the common WLS initialization, one-time setup, development
searches, and offline training/validation scoring.  In particular, the warmed
\SMax\ times exclude 1.49 seconds of Numba compilation.  Target summaries are
medians of the seed-specific first stored checkpoints below 5\%; all 20 seeds
eventually cross.  \SMax\ and SGM report uniform averages
of pre-update iterates; LSE reports its best charged training incumbent;
SOCP reports its best feasible incumbent.

The optimization runs used Python 3.12.7, NumPy 1.26.0, SciPy 1.15.3, and
Numba 0.63.1 on Windows 11.  SGM times are
medians of three repeats, \SMax\ uses the 20 confirmatory seeds, and LSE and
SOCP each have one stored run.  Thread-count environment variables were not fixed,
so wall-time ratios are machine-specific secondary evidence.

\section{Private-Inference Experiment Details}
\label{app:experiment}

\subsection{Centered-logit finite maximum}

We use all 50,000 training and 10,000 test images of CIFAR-100
\citep{Krizhevsky2009CIFAR}. The public checkpoint is
\texttt{cifar100\_resnet20}. It implements the residual architecture of
\citet{HeEtAl2016ResNet} and is available at
\url{https://github.com/chenyaofo/pytorch-cifar-models}
(accessed August 4, 2026).
The exposed module is the post-addition ReLU at the final residual block.  Its preactivation for
image \(s\), channel \(j\), and spatial position \(u\) is \(z_{sju}\).  Let
\(q=14\), let \(T_r\) be the degree-\(r\) Chebyshev polynomial, and set
\[
    B=1.10\max_{s,j,u\text{ in training}}|z_{sju}|,
    \qquad
    p_y(z)=B\sum_{r=0}^q y_rT_r(z/B).
\]
The 10\% margin is fixed before test evaluation.

Let \(w_{cj}\) be the fixed final-classifier weights and
\(|\mathcal U|=64\).  Define the original activation contribution and the
basis contributions by
\begin{align*}
    a_{sc}
    &=\sum_{j=1}^{64}w_{cj}\frac1{|\mathcal U|}
      \sum_{u\in\mathcal U}\operatorname{ReLU}(z_{sju}),\\
    v_{scr}
    &=B\sum_{j=1}^{64}w_{cj}\frac1{|\mathcal U|}
      \sum_{u\in\mathcal U}T_r(z_{sju}/B).
\end{align*}
Centering over the 100 classes gives
\[
    \phi_{scr}=v_{scr}-\frac1{100}\sum_{\ell=1}^{100}v_{s\ell r},
    \qquad
    b_{sc}=a_{sc}-\frac1{100}\sum_{\ell=1}^{100}a_{s\ell}.
\]
The training objective is
\begin{equation}
    F_{\rm train}(y)
    =\max_{\substack{s=1,\ldots,50{,}000\\c=1,\ldots,100}}
      |\langle\phi_{sc},y\rangle-b_{sc}|,
    \label{eq:app-private-max}
\end{equation}
with \(d=15\) and \(N=5{,}000{,}000\).  Class centering removes a
sample-dependent common logit shift and therefore changes neither softmax
probabilities nor the predicted class.

\subsection{Functional safeguard}

A dense exchange LP on 65,537 scalar points constructs a near-minimax
degree-14 approximation \(y_{\rm fun}\) to ReLU on \([-1,1]\).  Its
grid maximum is padded between adjacent nodes using
\(1+\sum_r|y_{{\rm fun},r}|r^2\), a global Lipschitz bound for the scalar
error.  Calling the result \(e_{\rm fun}\), we optimize over
\begin{equation}
    \mathcal Y=\left\{y:
      |y_r-y_{{\rm fun},r}|\le\frac{\kappa e_{\rm fun}}{q+1},
      \ r=0,\ldots,q\right\},
    \qquad \kappa=4.
    \label{eq:app-functional-box}
\end{equation}
Since \(|T_r(x)|\le1\) on \([-1,1]\), every feasible polynomial satisfies
\[
    \sup_{|z|\le B}|p_y(z)-\operatorname{ReLU}(z)|
    \le(1+\kappa)Be_{\rm fun}.
\]
In the recorded run, \(B=70.9681\), \(e_{\rm fun}=1.00987\times10^{-2}\),
the coefficient-box halfwidth is \(2.69299\times10^{-3}\), and the physical
scalar-error bound is 3.58344.  The largest held-out activation has
\(|z|/B=0.8389\); none of the 40,960,000 held-out activation values leaves the
fitted interval.

\subsection{Optimizers and protocol}
\label{app:experiment-protocol}

Bounded weighted least squares forms the \(15\times15\) normal equations in
one streamed pass and supplies the common warm start.  Restricted LP exchange
solves signed epigraph subproblems and performs exact streamed separation over
all five million absolute residuals \citep{ZhangEtAl2010Exchange}; each
restricted master is a linear Chebyshev fit \citep{BarrodalePhillips1975}.
LSE smooths both signed residuals by LogSumExp
\citep{Nesterov2005Smoothing}; L-BFGS-B enforces the common coefficient box
\citep{ByrdEtAl1995LBFGSB}.

The restricted LP gives the numerical lower reference
\(\ell=1.6493122505\), and all plotted gaps use
\begin{equation}
    \operatorname{relgap}_\ell(y)
    =\left[\frac{F_{\rm train}(y)-\ell}{\ell}\right]_+.
    \label{eq:app-private-gap}
\end{equation}
The final master and separation values agree to about \(1.2\times10^{-7}\)
in objective units.  We nevertheless describe \(\operatorname{relgap}_\ell\)
as a gap to
a numerical LP reference rather than a formal certificate.

\SMax\ is selected using seeds 0--2; every candidate run is capped at the
common development and production horizon
\(\mathsf Q_{\rm dev}=25N=1.25\times10^8\) and uses
\(\eta_t=\eta_0/\sqrt{t+t_0}\).  The training-only search selects
\(m=8192\), \(\eta_1=2.5192\times10^{-6}\), \(t_0=2128.895\), and
\(\eta_0=1.1626\times10^{-4}\).
This configuration is frozen and evaluated on confirmatory seeds 100--119.
Full-grid SGM and LSE are tuned by the same endpoint training criterion at the
same \(25N\) budget.  SGM uses the
unshifted schedule \(\eta_t=3.1490\times10^{-7}/\sqrt t\).  LSE uses
\(\mu/F_{\rm train}(y_{\rm WLS})
\in\{3\times10^{-2},10^{-2},3\times10^{-3},10^{-3}\}\)
and at most eight L-BFGS-B iterations per stage.  LP exchange terminates after
two full separation scans.

One absolute residual counts as one component value.  Query and runtime axes
exclude the common WLS fit, scalar-anchor construction, training-only tuning,
random-access setup, warm-up, and offline train/test scoring.  The common WLS
fit streams all \(N\) rows once.  The \SMax\ target summaries are medians of
seed-specific first stored crossings; all 20 seeds eventually cross.  Its
reported endpoints are componentwise medians of seed-specific outputs,
not metrics of coefficients averaged across seeds.  The primary held-out
metric is the maximum centered-logit distortion; accuracy and agreement with
the original model are secondary checks.

Optimization used the Python, NumPy, SciPy, and operating system
reported above.  SGM and LSE times are medians of two repeats, \SMax\ uses the
20 confirmatory seeds, and LP has one stored run; the same machine-specific
timing qualification applies.  Feature extraction used PyTorch
2.6.0.dev20241112 and Torchvision 0.20.0.dev20241112 with deterministic
algorithms on an NVIDIA GeForce GTX 1050 Max-Q GPU, but extraction time is not
included in optimizer timing.

Also, 12 of 15 WLS coefficients and 13 of 15 LP coefficients are at a box
boundary.

\subsection{Pointwise near-activity diagnostic}
\label{app:near-active-diagnostic}

At every stored coefficient vector \(x\), we compute all training components,
sort them, and define
\[
    \widehat k_\varepsilon(x)
    =\max\{k:F_{\rm train}(x)-\operatorname{AT}_k(x)\le\varepsilon/4\},
    \qquad \varepsilon=0.05\ell,
\]
where \(\ell\) is the corresponding numerical lower reference.  Table
\ref{tab:effective-k} reports medians and IQRs across confirmatory seeds.

\begin{table}[!htb]
    \caption{Pointwise effective number of near-active components.  Brackets
    show the IQR across 20 \SMax\ confirmatory runs.}
    \label{tab:effective-k}
    \centering
    \footnotesize
    \setlength{\tabcolsep}{5pt}
    \begin{tabular}{@{}lccc@{}}
        \toprule
        Problem & WLS start & First 5\% crossing & \(25N\) endpoint \\
        \midrule
        VFD & 2 & 85.5 [79.75, 88.25] & 92 [13.5, 127] \\
        Polynomial calibration & 1 & 1 [1, 1] & 3 [3, 3] \\
        \bottomrule
    \end{tabular}
\end{table}
\FloatBarrier

For comparison, the converged SOCP and LP reference points have
\(\widehat k_\varepsilon=109\) and 5, respectively.  Thus the VFD trajectory
shows many near-active components precisely in the moderate-accuracy regime.
The polynomial diagnostic remains small, so the sufficient top-\(k\) condition
alone does not explain that experiment's speedup.  These measurements are
pointwise evidence at stored iterates, not a numerical verification of the
supremum over the entire sublevel set in~\eqref{eq:localized-topk}.  The extra
full scans are offline diagnostics and are excluded from optimizer cost.

\end{document}